\documentclass[12pt]{amsart}

\usepackage{hyperref}
\usepackage{amsmath}
\usepackage{amssymb}
\usepackage{mathrsfs}
\usepackage{amsthm}
\usepackage{svninfo}

\usepackage[english]{babel}

\makeatletter

\ifx\Presentation\undefined
\fi

\def\newrefformat#1#2{%
  \@namedef{pr@#1}##1{#2}}
\def\fref#1{\@prettyref#1:}
\def\@prettyref#1:#2:{%
  \expandafter\ifx\csname pr@#1\endcsname\relax%
    \PackageWarning{prettyref}{Reference format #1\space undefined}%
    \ref{#1:#2}%
  \else%
    \csname pr@#1\endcsname{#1:#2}%
  \fi%
}

\newrefformat{eq}{\eqref{#1}}
\newrefformat{cha}{Chapter~\ref{#1}}
\newrefformat{sec}{Section~\ref{#1}}
\newrefformat{apx}{Appendix~\ref{#1}}
\newrefformat{tab}{Table~\ref{#1} on page~\pageref{#1}}
\newrefformat{fig}{Figure~\ref{#1} on page~\pageref{#1}}
\newrefformat{item}{\ref{#1}}

\ifx\thmnum\undefined
\newtheorem{thmnum}{}[section]
\fi

\newcommand{\mynewthm}[3][]{%
  \def\PARAM{#1}
  \ifx\PARAM\empty
  \newtheorem{#2}[thmnum]{#3}
  \else
  \newtheorem{#2}{#3}[#1]
  \fi
  \newtheorem*{#2*}{#3}%
  \newrefformat{#2}{#3~\ref{##1}}%
}

\ifx\l@french\undefined
\newcommand{\ThmLabel}{Theorem}
\newcommand{\PrpLabel}{Proposition}
\newcommand{\LemLabel}{Lemma}
\newcommand{\FctLabel}{Fact}
\newcommand{\CorLabel}{Corollary}
\newcommand{\DfnLabel}{Definition}
\newcommand{\ConvLabel}{Convention}
\newcommand{\NtnLabel}{Notation}
\newcommand{\CstLabel}{Construction}
\newcommand{\ExmLabel}{Example}
\newcommand{\RmkLabel}{Remark}
\newcommand{\QstLabel}{Question}

\else
\newcommand{\ThmLabel}{\iflanguage{french}{Th\'eor\`eme}{Theorem}}
\newcommand{\PrpLabel}{Proposition}
\newcommand{\LemLabel}{\iflanguage{french}{Lemme}{Lemma}}
\newcommand{\FctLabel}{\iflanguage{french}{Fait}{Fact}}
\newcommand{\CorLabel}{\iflanguage{french}{Corollaire}{Corollary}}
\newcommand{\DfnLabel}{\iflanguage{french}{D\'efinition}{Definition}}
\newcommand{\ConvLabel}{Convention}
\newcommand{\NtnLabel}{Notation}
\newcommand{\CstLabel}{Construction}
\newcommand{\ExmLabel}{\iflanguage{french}{Exemple}{Example}}
\newcommand{\RmkLabel}{\iflanguage{french}{Remarque}{Remark}}
\newcommand{\QstLabel}{Question}

\fi

\theoremstyle{plain}
\mynewthm{thm}{\ThmLabel}
\mynewthm{prp}{\PrpLabel}
\mynewthm{lem}{\LemLabel}
\mynewthm{fct}{\FctLabel}
\mynewthm{cor}{\CorLabel}

\theoremstyle{definition}
\mynewthm{dfn}{\DfnLabel}
\mynewthm{conv}{\ConvLabel}
\mynewthm{conj}{Conjecture}
\mynewthm{ntn}{\NtnLabel}
\mynewthm{cst}{\CstLabel}

\theoremstyle{remark}
\mynewthm{rmk}{\RmkLabel}
\mynewthm{qst}{\QstLabel}
\mynewthm{exm}{\ExmLabel}

\ifx\Presentation\undefined

\newcommand{\myenumlabel}[1]{\textnormal{(\roman{#1})}}

\fi

\renewcommand{\today}{%
  \number\day\space
  \ifcase\month\or
  January\or February\or March\or April\or May\or June\or
  July\or August\or September\or October\or November\or December\fi
  \space \number\year}

\newcounter{cycprfcnt}
\newenvironment{cycprf}%
{\begin{list}{\PackageWarning{begnac}{Label required for cycprf}}%
  {%
    \setcounter{cycprfcnt}{1}
    \setlength{\itemindent}{0.5\leftmargin}%
    \setlength{\leftmargin}{0pt}%
    \newcommand{\cpcurr}{\myenumlabel{cycprfcnt}}%
    \newcommand{\cpnext}{\addtocounter{cycprfcnt}{1}\cpcurr}%
    \newcommand{\cpnum}[1]{\setcounter{cycprfcnt}{##1}\cpcurr}%
    \newcommand{\cpfirst}{\cpnum{1}}%
    \newcommand{\impnext}{\cpcurr{} $\Longrightarrow$ \cpnext.}%
    \newcommand{\impfirst}{\cpcurr{} $\Longrightarrow$ \cpfirst.}%
  }%
}%
{\qedhere\end{list}}%

\def\indsym#1#2{%
  \setbox0=\hbox{$\m@th#1x$}%
  \kern\wd0%
  \hbox to 0pt{\hss$\m@th#1\mid$\hbox to 0pt{$\m@th#1^{#2}$\hss}\hss}%
  \lower.9\ht0\hbox to 0pt{\hss$\m@th#1\smile$\hss}%
  \kern\wd0}

\def\nindsym#1#2{%
  \setbox0=\hbox{$\m@th#1x$}%
  \kern\wd0%
  \hbox to 0pt{\hss$\m@th#1\not$\kern1.4\wd0\hss}
  \hbox to 0pt{\hss$\m@th#1\mid$\hbox to 0pt{$\m@th#1^{#2}$\hss}\hss}%
  \lower.9\ht0\hbox to 0pt{\hss$\m@th#1\smile$\hss}%
  \kern\wd0}

\def\dotminussym#1#2{%
  \setbox0=\hbox{$\m@th#1-$}%
  \kern.5\wd0%
  \hbox to 0pt{\hss\hbox{$\m@th#1-$}\hss}%
  \raise.6\ht0\hbox to 0pt{\hss$\m@th#1.$\hss}%
  \kern.5\wd0}
\newcommand{\dotminus}{\mathbin{\mathpalette\dotminussym{}}}

\renewcommand{\emptyset}{\varnothing}
\renewcommand{\setminus}{\smallsetminus}
\def\models{\vDash}

\newcommand{\rest}{{\restriction}}

\newcommand{\half}[1][1]{\hbox{$\frac{#1}{2}$}}

\DeclareMathOperator{\tp}{tp}

\DeclareMathOperator{\Th}{Th}
\DeclareMathOperator{\Mod}{Mod}
\DeclareMathOperator{\tS}{S}
\DeclareMathOperator{\dcl}{dcl}
\DeclareMathOperator{\acl}{acl}

\DeclareMathOperator{\diam}{diam}

\DeclareMathOperator{\id}{id}

\DeclareMathOperator{\Hom}{Hom}

\DeclareMathOperator{\Aut}{Aut}

\DeclareMathOperator{\Gal}{Gal}

\DeclareMathOperator{\dom}{dom}

\DeclareMathOperator{\Pert}{Pert}

\DeclareMathOperator{\CB}{CB}
\DeclareMathOperator{\RM}{RM}
\DeclareMathOperator{\dM}{dM}

\newcommand{\fA}{\mathfrak{A}}
\newcommand{\fB}{\mathfrak{B}}

\newcommand{\fM}{\mathfrak{M}}

\newcommand{\fp}{\mathfrak{p}}

\newcommand{\cB}{\mathcal{B}}

\newcommand{\cF}{\mathcal{F}}

\newcommand{\cL}{\mathcal{L}}

\newcommand{\cQ}{\mathcal{Q}}

\newcommand{\sB}{\mathscr{B}}

\newcommand{\sT}{\mathscr{T}}
\newcommand{\sU}{\mathscr{U}}

\newcommand{\bC}{\mathbb{C}}

\newcommand{\bQ}{\mathbb{Q}}
\newcommand{\bR}{\mathbb{R}}

\makeatother

\DeclareMathOperator{\wt}{wt}
\DeclareMathOperator{\CBd}{CBd}
\newcommand{\fT}{\mathfrak{T}}
\newcommand{\fTM}{\mathfrak{TM}}
\newcommand{\defcomp}{\chi_{\textrm{def}}}

\begin{document}

\title{Topometric spaces and perturbations of metric structures}

\author{Ita\"\i{} \textsc{Ben Yaacov}}

\address{Ita\"\i{} \textsc{Ben Yaacov} \\
  Universit\'e de Lyon \\
  Universit\'e Lyon 1 \\
  Institut Camille Jordan, CNRS UMR 5208 \\
  43 boulevard du 11 novembre 1918 \\
  F-69622 Villeurbanne Cedex \\
  France}

\urladdr{http\string://math.univ-lyon1.fr/\textasciitilde begnac/}

\thanks{Research partially supported by NSF grant DMS-0500172.}
\thanks{Work on the present paper started during the
  Methods of Logic in
  Mathematics III meeting at the Euler Institute in Saint-Petersburg, and
  formed the basis for a tutorial given
  at the Oxford Workshop in Model Theory, Oxford University, during
  summer 2006.
  The author wishes to thank the organisers of both conferences
  for the invitations.
  The author wishes to extend particular thanks to the
  tree at the corner
  of Kamennoostrovskii Prospekt and Pesochnaya Naverjnaya for kindly
  choosing to miss him.}
\thanks{The author would also like to thank Anand Pillay for stimulating
  discussions and the referees for many helpful comments.}

\svnInfo $Id: TopoPert.tex 589 2008-05-22 13:55:23Z begnac $
\thanks{\textit{Revision}: {\svnInfoRevision};
  \textit{Date}: \today}


\begin{abstract}
  We develop the general theory of \emph{topometric spaces}, i.e.,
  topological spaces equipped with a well-behaved lower
  semi-continuous metric.
  Spaces of global and local types in continuous logic are the
  motivating examples for the study of such spaces.

  In particular, we develop Cantor-Bendixson analysis of
  topometric spaces, which can serve as a basis for the study of local
  stability (extending the \textit{ad hoc} development in
  \cite{BenYaacov-Usvyatsov:CFO}),
  as well as of global $\aleph_0$-stability.

  We conclude with a study of perturbation systems (see
  \cite{BenYaacov:Perturbations}) in the formalism of topometric spaces.
  In particular, we show how the abstract development applies to
  $\aleph_0$-stability up to perturbation.
\end{abstract}

\maketitle

\section*{Introduction}

Topometric spaces, namely spaces equipped both with a topology and
with a metric,
are omnipresent in continuous logic and in fact predate it.

Global type spaces, in the sense of continuous logic, as well as in
the sense of predecessors such as Henson's logic or metric compact
abstract theories, are equipped with a logic topology as well as with
a natural metric
$d(p,q) = d\bigl( p(\bar M),q(\bar M) \bigr)$ (where $\bar M$ is
the monster model).
Iovino's notion of a uniform structure on the type
spaces \cite{Iovino:StableBanach} is an early attempt to
put this metric structure in a more general setting,
and as such may be viewed
as a precursor to the formalism we propose here.
The metric nature of global types spaces was used by Iovino and
later by the author to define useful notions of Morley ranks.
These ranks play a crucial role
in the proof of Morley's Theorem for metric structures in
\cite{BenYaacov:Morley}.

Continuous logic was proposed and developed in
\cite{BenYaacov-Usvyatsov:CFO} as a model-theoretic formalism for
metric structures.
Unlike its predecessors it provides a good notion of a local type
space $\tS_\varphi(M)$, namely the space of $\varphi$-types for a
fixed formula $\varphi$.
Again, in addition to the logic topology, this space is equipped with a
useful metric
$d_\varphi(p,q) = \sup_{\bar b} |\varphi(\bar x,\bar b)^p -
\varphi(\bar x,\bar b)^q|$.
More examples comes from the study of perturbations of metric
structures in \cite{BenYaacov:Perturbations}, where perturbation
metrics turn out to be
alternative topometric structures on the type spaces.

In addition, topometric analogues of the classical (one should say
``discrete'') Cantor-Bendixson analysis in these spaces play important
roles in various contexts.
In global type spaces they can be used to
characterise $\aleph_0$-stability and define the Morley ranks
which were constructed (in a far more complicated manner) in
\cite{BenYaacov:Morley}.
In local type spaces they can be use to
characterise local stability and independence, as in
\cite{BenYaacov-Usvyatsov:CFO}.
Finally, at the end of the present paper, we use them to
for a rudimentary study of
the notion of a theory being \emph{$\aleph_0$-stable up to perturbation},
which occurs more and more in recently studied examples.
In particular we show that this property is characterised by the
existence of corresponding Morley ranks.

In the present paper we unite these examples under the single
definition of a \emph{topometric space}.
We then proceed to study topometric spaces as such,
much like general topology studies topological spaces,
with a particular emphasis on Cantor-Bendixson analysis.
Alongside this abstract study we provide many motivating examples from
continuous logic as well as applications of our abstract results to
the study of metric structures.

In \fref{sec:Topo} we define topometric spaces and
the category of topometric spaces.
While we are quite certain about the category of \emph{compact}
topometric spaces, we propose to extend our definitions to locally
compact and even more general spaces, with some lower degree of
certitude.
In particular, we study questions such the existence of quotients
which preserve part of the structure, which seem to be a little more
complicated than for classical topological spaces.

In \fref{sec:Isol} we study various notions analogous to isolation in
classical topological spaces.
In the case of topometric type spaces, \emph{$d$-isolated} types are
indeed the correct analogues of isolated types in classical logic
(such types were referred to in \cite{BenYaacov:Morley} as
\emph{principal}, following Henson's earlier terminology).

In \fref{sec:CB} we study several natural notions of Cantor-Bendixson
ranks,
showing that they all give rise the same notion of Cantor-Bendixson
analysability.
We give several
results characterising Cantor-Bendixson analysability and comparing
the Cantor-Bendixson ranks of two spaces.

Finally, in \fref{sec:PertMet} we study the special case of topometric
spaces arising as type spaces equipped with perturbation metrics.
We study notions such as $\lambda$-stability, and in particular
$\aleph_0$-stability, up to perturbation.

For the purpose of examples we shall assume familiarity with the
basics of continuous first order logic, as developed in
\cite{BenYaacov-Usvyatsov:CFO}.
For a general survey of continuous logic and the model theory of
metric structures we refer the reader to
\cite{BenYaacov-Berenstein-Henson-Usvyatsov:NewtonMS}.

\section{Topometric spaces}
\label{sec:Topo}

\subsection{Basic properties}
While a metric is usually defined to take values in $[0,\infty)$, we
allow infinite distances.
If $(X,d)$ is a metric space then distances between sets are defined
as usual $d(A,B) = \inf \{d(x,y)\colon x \in A, y \in B\}$ and
$d(x,A) = d(\{x\},A)$.
We follow the convention that
$d(x,\emptyset) = d(A,\emptyset) = \inf \emptyset = \infty$.

\begin{ntn}
  Let $(X,d)$ be a metric space, $A \subseteq X$, $r \in \bR^+$.
  We define:
  \begin{align*}
    B(A,r) & = \{x \in X\colon d(x,A) < r\} \\
    {\overline B}(A,r) & = \{x \in X\colon d(x,A) \leq r\}
  \end{align*}
  When $A$ is a singleton $\{a\}$ we may write $B(a,r)$ and
  ${\overline B}(a,r)$ instead.
\end{ntn}

\begin{dfn}
  \label{dfn:Topometric}
  A \emph{(Hausdorff) topometric space} is a
  triplet $(X,\sT,d) = (X,\sT_X,d_X)$ where $X$ is a
  set of points, $\sT$ is a topology on $X$
  and $d$ is a $[0,\infty]$-valued metric on $X$:
  \begin{enumerate}
  \item The metric refines the topology.
    In other words, for every open $U \subseteq X$ and every $x \in U$ there is
    $r > 0$ such that $B(x,r) \subseteq U$.
  \item The metric function $d\colon X^2 \to [0,\infty]$
    is lower semi-continuous,
    i.e., for all $r \in \bR^+$ the set
    $\{(a,b) \in X^2\colon d(a,b) \leq r\}$ is closed in $X^2$.
  \end{enumerate}
\end{dfn}

\begin{conv}
  We shall follow the convention that unless explicitly qualified
  otherwise, terms and notations from the vocabulary of
  general topology (e.g., compactness, continuity, etc.)\
  refer to the topological space $(X,\sT)$, while
  terms from the vocabulary of metric spaces which are not applicable
  in general topology (e.g., uniform continuity, completeness)
  refer to the metric space $(X,d)$.
\end{conv}

The topological closure of a subset $Y \subseteq X$ is denoted
by $\overline Y$.
Note that the closed set
$\overline{B(a,r)}$ should not
be confused with ${\overline B}(a,r)$,
which is defined in pure metric terms.
Lower semi-continuity of $d$ implies that ${\overline B}(a,r)$
is closed (more generally, we show in \fref{lem:ClsdMetNeighb}
below that ${\overline B}(F,r)$ is closed for every compact $F$), so
$\overline {B(a,r)} \subseteq {\overline B}(a,r)$.
A discrete $0/1$ metric provides us with an extreme example
of proper inclusion:
$\overline {B(a,1)} = \{a\} \neq X = {\overline B}(a,1)$.

Recall that the \emph{weight} of a topological space $X$, denoted $\wt(X)$,
is the minimal cardinality of a base of open sets for $X$.
Similarly, if $X$ is a metric space, we use $\|X\|$ to denote its
\emph{density character}, i.e., the minimal size of a dense subset.
In case $X$ is a topometric space, $\wt(X)$ refers to its topological
part while $\|X\|$ to its metric part.
If $X$ is a finite space then $\|X\| = \wt(X) = |X|$; otherwise both are
infinite.

\begin{lem}
  \label{lem:LowerSemiContinuity}
  A topological space $X$ equipped with a lower semi-continuous
  metric is Hausdorff.
  If $X$ is in addition compact then the metric must refine the
  topology and $X$ is a topometric space
  (in other words, for compact spaces the second item in
  \fref{dfn:Topometric} implies the first item).
\end{lem}
\begin{proof}
  Assume $X$ is equipped with a lower semi-continuous metric $d$.
  Then the diagonal $\Delta X = d^{-1}(0) \subseteq X \times X$
  is closed and $X$ is Hausdorff.

  We observed earlier that ${\overline B}(a,r)$ is closed for all $a \in X$ and
  $r \in \bR^+$.
  If $U$ is a neighbourhood of $a$ then
  $\bigcap_{r>0} {\overline B}(a,r) = \{a\} \subseteq U$.
  If $X$ is compact then
  ${\overline B}(a,r) \subseteq U$ for some $r> 0$,
  so the metric refines the topology, as desired.
\end{proof}

\begin{exm}
  The motivating examples come from
  continuous first order logic \cite{BenYaacov-Usvyatsov:CFO}.
  Type spaces, which are naturally equipped with an intrinsic
  ``logic topology'', also admit one or several metric structures
  rendering them topometric spaces:
  \begin{enumerate}
  \item The type spaces $\tS_n(T)$ of a theory $T$, equipped with what
    we call the \emph{standard metric}:
    $$d(p,q) = \inf \{d(a,b)\colon \text{$a \models p$ and $b \models q$ in a
      saturated model $M \models T$}\}.$$
    If $T$ is incomplete and $p,q$ belong to distinct completions then
    $d(p,q) = \infty$.
  \item Type spaces in unbounded continuous logic were defined in
    \cite{BenYaacov:Perturbations},
    with the distance between types defined as above.
    Unlike type spaces in standard continuous logic, these are merely
    locally compact.
  \item Local type spaces $\tS_\varphi(M)$ over a model:
    $$d(p,q) = \sup \{|\varphi(x,b)^p-\varphi(x,b)^q|\colon b \in M\}.$$
  \item Perturbation systems are presented in \cite{BenYaacov:Perturbations} via an
    alternative system of topometric structures on the type spaces
    $\tS_n(T)$, where the metric is the ``perturbation distance''
    $d_\fp$.
    Here $d(p,q) = \infty$ means that a realisation of $p$ cannot be
    perturbed into a realisation of $q$.
  \item The metric $\tilde d_\fp$, defined in \cite{BenYaacov:Perturbations}
    as a combination of $d$ and $d_\fp$, also renders the topological
    space $\tS_n(T)$ a topometric space.
  \end{enumerate}
\end{exm}

There are two extreme kinds of topometric spaces which arise naturally
from standard topological and metric spaces:
\begin{dfn}
  \begin{enumerate}
  \item A \emph{maximal} topometric space is one in which the metric
    is discrete.
  \item A \emph{minimal} topometric space is one in which the metric
    coincides with the topology.
  \end{enumerate}
\end{dfn}

\begin{exm}
  Every Hausdorff topological space can be naturally
  viewed as a maximal topometric space.
  Similarly, every metric space can be naturally viewed as a minimal
  topometric space.
\end{exm}

Clearly a topometric space $X$ is minimal if and only if
the metric function $d\colon X^2 \to \bR^+$ is continuous
(rather than merely lower semi-continuous).
If $X$ is compact, then this is further equivalent to the metric
topology on $X$ being compact.

A topometric space is both minimal and maximal if and only if it is
topologically discrete: thus the minimal topometric spaces
should be viewed as the topometric generalisation of
classical discrete topological spaces.

\begin{lem}
  \label{lem:ClsdMetNeighb}
  Let $X$ be a topometric space, $F \subseteq X$ compact.
  Then ${\overline B}(F,r)$ is closed for every $r \in \bR^+$
  (we say that $F$ has closed metric neighbourhoods in $X$).
\end{lem}
\begin{proof}
  For $r \in \bR^+$ let $G_r = F \times X \cap d^{-1}([0,r])$.
  Then $G_r \subseteq F \times X$ is closed.
  Since $F$ is compact, the projection $\pi\colon F \times X \to X$ is closed,
  whereby $\pi(G_r)$ is closed.
  We conclude that ${\overline B}(F,r) = \bigcap_{r' > r} \pi(G_{r'})$ is closed.
\end{proof}

In particular, if $X$ is compact then every closed subset of $X$ has
closed metric neighbourhoods.
In fact we can say something slightly stronger:

\begin{lem}
  \label{lem:TopometricDef}
  Let $X$ be a Hausdorff topological space, $d$ a metric on
  $X$.
  If $X$ is compact and $d\colon X^2 \to [0,\infty]$
  is lower semi-continuous then
  every closed set in $X$ has
  closed metric neighbourhoods.
  Conversely, if $X$ is regular (in particular, if $X$ is locally
  compact) and every closed set has closed metric neighbourhoods then
  $d$ is lower semi-continuous.
\end{lem}
\begin{proof}
  The first assertion is a consequence of
  \fref{lem:ClsdMetNeighb}.
  For the converse assume ${\overline B}(F,r)$ is
  closed for every closed $F$ and $r > 0$.
  Assume that $d(x,y) > r$, and choose
  some intermediate values $d(x,y) > r_1 > r_2 > r$.
  First, we have $x \notin {\overline B}(y,r_1)$ and the latter is closed.
  We can therefore find an open set $V$ such that
  $x \in V \subseteq \overline V \subseteq X \setminus {\overline B}(y,r_1)$,
  whereby $d(\overline V,y) \geq r_1 > r_2$, and thus
  $y \notin {\overline B}(\overline V,r_2)$.
  Following the same reasoning we find $U$ open such that
  $y \in U \subseteq
  \overline U \subseteq X \setminus {\overline B}(\overline V,r_2)$.
  We conclude that
  $d(V,U) \geq d(\overline V,\overline U) \geq r_2 > r$, and
  $(x,y) \in V\times U \subseteq X^2 \setminus d^{-1}([0,r])$,
  as desired.
\end{proof}

Along with \fref{lem:LowerSemiContinuity},
this means that our definition of a \emph{compact} topometric space
coincides with that given in \cite{BenYaacov-Usvyatsov:CFO},
based on closed sets having closed metric neighbourhoods.
If we drop the compactness assumption then still in every
``reasonable'' space lower semi-continuity of the metric is weaker
than the closed metric neighbourhoods assumption.
It is nonetheless sufficient for our purposes in compact and locally
compact spaces, and as all natural examples are at least locally
compact we have no choice but to base our intuition on those.
Finally, lower semi-continuity passes to product spaces (equipped with
the supremum distance) whereas the closed metric neighbourhoods
property does not seem to do so.

\begin{lem}
  \label{lem:DistFromCompact}
  Let $X$ be a topometric space, $K,F \subseteq X$ compact.
  Then $d(K,F) = \min \{d(x,y)\colon x\in K,y\in F\}$, i.e., the minimum is
  attained by some $x_0 \in K$ and $y_0 \in Y$.
\end{lem}
\begin{proof}
  Let $r = d(K,F) = \inf \{d(x,y)\colon x\in K,y\in F\}$.
  Then for every $r' > r$: 
  $(K\times F) \cap d^{-1}([0,r']) \neq \emptyset$.
  As $d^{-1}([0,r]) = \bigcap_{r'>r} d^{-1}([0,r'])$ and this is a decreasing
  intersection of compact sets we get
  $(K\times F) \cap d^{-1}([0,r]) \neq \emptyset$, as requried.
\end{proof}

\begin{prp}
  \label{prp:CompactComplete}
  Every compact topometric space is complete.
\end{prp}
\begin{proof}
  Let $\{x_n\}_{n<\omega}$ be a Cauchy sequence in $X$.
  For each $n$ let $r_n = \sup \{d(x_n,x_m)\colon m > n\}$, so $r_n \searrow 0$.
  Then ${\overline B}(x_n,r_n)$ is closed for all $n$ (by
  \fref{lem:ClsdMetNeighb}) and contains $x_m$ for all $m \geq n$.
  By compactness
  $F = \bigcap_{n < \omega} {\overline B}(x_n,r_n) \neq \emptyset$.
  It follows that $F = \{x\}$ where $x$ is the metric limit of
  $\{x_n\}_{n < \omega}$.
\end{proof}

On the other hand, for non-compact spaces completeness is not
isomorphism-invariant: indeed, the minimal spaces $(0,1)$ and $\bR$ are
isomorphic as (minimal) topometric spaces,
yet only one of them is complete.
A more interesting example is that of type spaces of a theory $T$ in
unbounded continuous logic (see \cite{BenYaacov:Perturbations}):
Each type space $\tS_n(T)$
is locally compact and complete.
The compactification procedure
described there consists of embedding it $\tS_n(T) \hookrightarrow \tS_n(T^\infty)$,
where $\tS_n(T^\infty)$ is a compact type space of a standard (i.e., bounded)
theory $T^\infty$.
This embedding is a morphism (continuous and locally uniformly
continuous) and the image is incomplete -- indeed, it is metrically
dense in $\tS_n(T^\infty)$.

\begin{qst}
  Can a ``single point'' (or at least ``few points'') compactification
  be constructed for locally compact topometric spaces?
  In the case of an unbounded continuous theory $T$,
  we should like $\tS_n(T^\infty)$ to be a ``few points compactification''
  of $\tS_n(T)$.
\end{qst}

\subsection{The category of topometric spaces}
In this paper we deal mostly with compact or locally compact
topometric spaces (by our convention this means the topology is
compact or locally compact).
The correct notion of a morphism of compact topometric
spaces seems clear.

\begin{dfn}
  \label{dfn:CompactMorphism}
  Let $X$ and $Y$ be topometric spaces, $X$ compact.
  A \emph{morphism} $f\colon X \to Y$ is a mapping
  which is both continuous (topologically) and uniformly continuous
  (metrically).
\end{dfn}

We seek a candidate for the definition of a morphism between general
topometric spaces.
In the non compact case uniform continuity seems too strong
a requirement.
\begin{dfn}
  \label{dfn:LocUnifCont}
  Let $f\colon X\to Y$ be a mapping between topometric spaces.
  We define some properties of $f$ which depend both on the
  topological and the metric structures of $X$:
  \begin{enumerate}
  \item We say that $f$ is \emph{locally uniformly continuous} if for
    every $x \in X$ has a neighbourhood on which $f$ is uniformly continuous.
  \item We say that $f$ is \emph{weakly locally uniformly continuous} if for
    every $x \in X$ and $\varepsilon > 0$ there are a neighbourhood $U$ of $x$
    and $\delta > 0$ such that if $y,y' \in U$ and
    $d_X(y,y') < \delta$ then $d_Y(f(y),f(y')) \leq \varepsilon$.
  \item We say that $f$ is \emph{uniformly continuous on every
      compact} if every restriction of $f$ to a compact subset of $X$
    is uniformly continuous.
  \end{enumerate}
\end{dfn}

\begin{lem}
  \label{lem:LocUnifCont}
  The properties defined in \fref{dfn:LocUnifCont} imply one another
  from top to bottom.
\end{lem}
\begin{proof}
  We only need to prove that if $X$ is compact and
  $f\colon X\to Y$ is weakly locally uniformly continuous then $f$ is
  uniformly continuous.

  Indeed, let $\varepsilon > 0$.
  Then $X$ admits an open cover
  $X = \bigcup_{i\in I} U_i$ such that for each $U_i$
  there is $\delta_i > 0$ such that
  if $x,x' \in U_i$ and $d_X(x,x') < \delta_i$ then
  $d_Y(f(x),f(x')) \leq \varepsilon$.
  The set $\bigcup_{i\in I} U_i \times U_i$
  is a neighbourhood of the diagonal of $X$,
  which is equal to $\bigcap_{\delta>0} d^{-1}([0,\delta])$.
  By definition of a topometric space each $d^{-1}([0,\delta])$ is closed.
  Thus by compactness we may assume that
  $\bigcup_{i<n} U_i \times U_i \supseteq d^{-1}([0,\delta'])$.
  Now let $\delta = \min \{\delta_i\colon i < n\}\cup\{\delta'\}$.
  If $x,x' \in X$ and $d(x,x') < \delta$ then
  $x,x' \in U_i$ for some $i < n$ and therefore
  $d(f(x),f(x')) \leq \varepsilon$.
\end{proof}

Thus for locally compact $X$ all these properties agree.
For non locally compact $X$ even local uniform continuity seems too
strong (e.g., for \fref{prp:AutMorph} below).
On the other hand, uniform continuity on every compact is too weak
(let $X$ be the minimal space based on $\bR \setminus \bQ$ and $Y$ the
maximal one: then $\id\colon X\to Y$ is uniformly continuous on every compact
while not being even metrically continuous.)
We find ourselves led to suggesting the intermediary property as the
definition.
\begin{dfn}[Tentative]
  \label{dfn:Morphism}
  Let $X$ and $Y$ be topometric spaces.
  A \emph{morphism} $f\colon X \to Y$ is a mapping
  which is both continuous and weakly locally uniformly continuous.
\end{dfn}
We leave it to the reader to check that the composition of two
morphisms is indeed a morphism.
By \fref{lem:LocUnifCont} this definition agrees with
\fref{dfn:CompactMorphism}.

\begin{prp}
  \label{prp:AutMorph}
  Let $f\colon X\to Y$ be a continuous mapping between topometric spaces, and
  assume that either $X$ is maximal or $Y$ is minimal.
  Then $f$ is a morphism.
\end{prp}
\begin{proof}
  The case where $X$ is maximal is immediate, so we prove the case
  where $Y$ is minimal.
  We need to show that $f$ is weakly locally uniformly continuous.
  So let $\varepsilon > 0$ and $x \in X$.
  Let $V = B(f(x),\varepsilon/2) \subseteq Y$.
  As $Y$ is minimal, $V$ is open, so $U = f^{-1}(V)$ is open in $X$.
  Then $x \in U$, and
  for every $y,y' \in U$ we have $d(f(y),f(y')) \leq \varepsilon$
  (no need for $\delta$ here).
\end{proof}

This justifies the terminology: a maximal topometric structure is the
strongest possible such structure on a topological space, while a
minimal structure is the weakest possible on a (metrisable)
topological space.

We get another reassurance about our definition of a morphism of
non-compact topometric spaces from the following result, telling us
that for all intents and purposes we may identify classical Hausdorff
topological spaces with maximal topometric spaces and classical metric
spaces with minimal topometric spaces:

\begin{prp}
  The construction of a maximal topometric space from a
  Hausdorff topological space is a functor $\fT \to \fTM$, from the
  category of Hausdorff spaces and continuous mappings to that of
  topometric spaces.
  As such it is the left-adjoint of the forgetful functor
  $\fT\fM \to \fT$.
  This functor is an equivalence of categories between
  Hausdorff topological spaces and maximal
  topometric spaces.

  Similarly the construction of a minimal topometric space from a metric
  one is a functor $\fM \to \fTM$, from the category of metric
  spaces and (metrically) continuous mappings to topometric spaces.
  It is the right-adjoint of the forgetful functor $\fT\fM \to \fM$,
  and is an equivalence of categories between metric spaces and
  minimal topometric spaces.
\end{prp}
\begin{proof}
  Let $\psi_\fT \colon \fT \to \fT\fM$ be the maximal topometric space
  construction and $\psi_\fM \colon \fM \to \fT\fM$ be the minimal topometric
  space construction.
  Both are functors by \fref{prp:AutMorph}.
  Let also $\varphi_\fT \colon \fT \to \fT\fM$
  and $\varphi_\fM \colon \fM \to \fT\fM$ be the
  forgetful functors.
  It is immediate from the definition that $\varphi_\fT$
  is indeed a functor; it is also not difficult to check that $\varphi_\fM$
  is.

  Let $X \in \fT$, $Y \in \fT\fM$, and $f\colon X \to Y$ a mapping between their
  underlying sets.
  Clearly if $f\colon \psi_\fT(X) \to Y$ is a morphism then
  $f\colon X \to \varphi_\fT(Y)$ is
  continuous, and the converse is by \fref{prp:AutMorph}.
  Thus $\psi_\fT$ is the left-adjoint of $\varphi_\fT$.
  If $X,Y \in \fT$ we get
  $\Hom_\fT(X,Y) = \Hom_\fT(X,\varphi_\fT\circ\psi_\fT(Y))
  = \Hom_{\fT\fM}(\psi_\fT(X),\psi_\fT(Y))$ whence the equivalence of categories.

  The argument for minimal topometric spaces and metric spaces is
  similar.
\end{proof}

\begin{dfn}
  Let $X$ be a topometric space.
  A \emph{(topometric) subspace} of $X$ is a subset $Y \subseteq X$ equipped
  with the induced structure.
  One easily verifies that this is indeed a topometric space (this was
  in fact used implicitly in the proof of \fref{lem:LocUnifCont}) and
  that the inclusion mapping $Y \hookrightarrow X$ is a morphisms.

  A mapping $f\colon Y \to X$ is a \emph{monomorphism} if it is an
  isomorphism with a subspace of $X$ (its image).
\end{dfn}

When dealing with quotients we feel more secure restricting to
compact spaces.
Recall that a continuous surjective mapping from a compact space to a
Hausdorff space is automatically a topological quotient mapping, so we
only need to worry about the metric structure.

\begin{lem}
  \label{lem:TopologicalQuotient}
  Let $(X,d_X)$ be a compact topometric space and
  $\pi\colon X \to Y$ a Hausdorff topological quotient of $X$.
  Let $d_Y(y,y') = d_X\bigl( \pi^{-1}(y),\pi^{-1}(y') \bigr)$.
  Then $(Y,d_Y)$ is a topometric space and
  $\pi\colon (X,d_X) \to (Y,d_Y)$ is a contractive homomorphism of
  topometric spaces.
\end{lem}
\begin{proof}
  For $r \geq 0$ let
  $\Delta^r_X = \{ (x,x') \in X^2\colon d_X(x,x') \leq r \}$
  and define $\Delta^r_Y \subseteq Y^2$ similarly.
  The set $\Delta^r_X$ is closed in $X^2$ by the lower semi-continuity
  of $d_X$ and by \fref{lem:DistFromCompact} we have
  $\Delta^r_Y = (\pi,\pi)(\Delta^r_X)$ so $\Delta^r_Y$ is closed as
  well.
  Therefore $d_Y$ is lower semi-continuous.
  By \fref{lem:TopometricDef} $d_Y$ refines the topology on $Y$ and
  $(Y,d_Y)$ is a topometric space.
  The rest follows directly from the construction.
\end{proof}

For many purposes we shall require a somewhat stronger notion
of quotient.

\begin{dfn}
  Let $X$ and $Y$ be topometric spaces, $X$ compact.
  An \emph{epimorphism} $f\colon X \to Y$ is a surjective morphism
  satisfying  that for every $\varepsilon > 0$
  there is $\delta > 0$ such that for
  all $x \in X$ and $y \in Y$:
  \begin{gather*}
    d_Y(f(x),y) < \delta \Longrightarrow d_X(x,f^{-1}(y)) \leq \varepsilon.
  \end{gather*}
  We then say that $Y$ is a \emph{quotient space} of $X$, and that $f$
  is a \emph{quotient mapping}.
\end{dfn}
Note that this property is in some sense a converse of uniform
continuity, which may be stated as:
\begin{gather*}
  (\forall\varepsilon)(\exists\delta)(\forall xy)
  \bigl( d_X(x,f^{-1}(y)) < \delta
  \Longrightarrow d_Y(f(x),y) \leq \varepsilon \bigr).
\end{gather*}
Also, a mapping is a surjective monomorphism if and only if it
is an injective epimorphism if and only if it is an isomorphism.

\begin{dfn}
  Let us call a mapping between metric spaces
  $f\colon X \to Y$ is \emph{precise} if for every $x \in X$ and $y \in Y$:
  $d_Y(f(x),y) = d_X(x,f^{-1}(y))$.
  We follow the convention that $d(x,\emptyset) = \infty$.
\end{dfn}

\begin{lem}
  \label{lem:PreciseMapping}
  Let $f\colon X \to Y$ be a continuous precise mapping between topometric
  spaces.
  Then:
  \begin{enumerate}
  \item It is uniformly continuous, and in particular a morphism.
  \item If $f$ is surjective and $X$ is compact then $f$ is an
    epimorphism.
    (Any generalisation of the definition of epimorphisms to
    non-compact spaces should preserve this property.)
  \item The distance from $f(X)$ to its complement $Y \setminus f(X)$ is
    infinite.
    Thus, if we exclude in $Y$ the infinite distance then a precise
    mapping is necessarily surjective.
  \item If $f$ is injective then it is isometric.
  \end{enumerate}
\end{lem}
\begin{proof}
  Immediate.
\end{proof}

\subsection{Existence of precise quotients }

We shall need a result saying that interesting quotients of compact
topometric spaces exist.
By interesting we mean ones which preserve some prescribed piece of
information without preserving too much else.
Throughout this subsection $X$ is a compact topometric space
(much of what we say can be extended to locally
compact spaces in a straightforward manner).

First, let us describe precise topometric quotients of $X$.

\begin{lem}
  \label{lem:PrecQuotRel}
  Let $\sim$ be an equivalence relation on $X$.
  Let $[x] = \{x' \in X\colon x \sim x'\}$ and 
  $Y = X/{\sim} = \{[x]\colon x \in X\}$ be the quotient set,
  $\pi\colon X \twoheadrightarrow Y$ the projection
  map.
  Then the following are equivalent:
  \begin{enumerate}
  \item The relation  $\sim$ is closed and for all $x,y \in X$:
    $d_X(x,[y]) = d_X([x],y)$.
  \item There exists a topometric structure on $Y$
    such that $\pi$ is a precise quotient map.
  \end{enumerate}
  Moreover, if such a topometric structure on $Y$ exists then it is
  unique, given by the quotient topology in $Y$ and
  the metric:
  \begin{gather*}
    d_Y([x],[y]) = d_X([x],[y]) = \inf \{d_X(x',y')\colon x'\sim x, y'\sim y\}.
  \end{gather*}
\end{lem}
\begin{proof}
  Assume first $\sim$ is closed and $d_X(x,[y]) = d_X([x],y)$.
  Then $Y$ is compact and Hausdorff, and that each equivalence class
  is compact.
  By \fref{lem:DistFromCompact}, for all $x,y \in X$ there are
  $x' \in [x]$ and $y' \in [y]$ such that
  $d_X([x],[y]) = d_X(x',y')$.
  We obtain that $[x] \neq [y] \Longrightarrow d_X([x],[y]) > 0$ and
  $d_X([x],[y]) = d_X(x',[y])
  = d_X([x],y) = d_X(x,[y])$, whereby
  $d_X([x],[z]) \leq d_X([x],y) + d_X(y,[z])
  = d_X([x],[y]) + d_X([y],[z])$.
  Thus $d_Y([x],[y]) = d_X([x],[y])$ is a metric.

  For all $r \in \bR^+$ the set
  $\{(x,y)\colon d_X(x,[y]) \leq r\} \subseteq X^2$ is closed as the projection
  on the first and last coordinate of
  $\{(x,y,z)\colon d_X(x,y) \leq r \text{ and } y\sim z\} \subseteq X^3$.
  Thus $\{([x],[y])\colon d_Y([x],[y]) \leq r\} \subseteq Y^2$
  is compact and therefore
  closed, and $(Y,\sT_Y,d_Y)$ is a topometric structure.
  The property $d_Y([x],[y]) = d_X(x,[y])$ also implies that
  $\pi$ is precise.

  Conversely, assume that $\pi$ is precise.
  Then it is immediate that the topometric structure on $Y$ is as
  prescribed in the moreover part, and that for all $x,y \in X$:
  \begin{gather*}
    d_X(x,[y]) = d_Y([x],[y]) = d_X([x],y).
    \qedhere
  \end{gather*}
\end{proof}

\begin{lem}
  \label{lem:PrecQuotRel2}
  Let $\sim$ be a closed equivalence relation on $X$.
  Then the following are equivalent:
  \begin{enumerate}
  \item For all $x,y \in X$: $d(x,[y]) = d([x],y)$.
  \item For every class $[x] \subseteq X$ and $r \in \bR^+$:
    ${\overline B}([x],r)$ is closed under $\sim$.
  \item For every closed set $F \subseteq X$ and $r \in \bR^+$, if $F$ is closed
    under $\sim$ then so is ${\overline B}(F,r)$.
  \item \label{item:PrecQuotRel2Family}
    There exists a family $B$ of closed sets in $X$ such that:
    \begin{enumerate}
    \item $B$ is closed under finite intersections.
    \item For all $x \not\sim y$ there is $F \in B$ such that
      $x \in F$, $y \notin F$.
    \item For all $F \in B$ and $r \in \bQ^+$: ${\overline B}(F,r)$ is closed under
      $\sim$.
    \end{enumerate}
  \end{enumerate}
\end{lem}
\begin{proof}
  \begin{cycprf}
  \item[\impnext] For all $y\sim y'$ we have
    $d(y,[x]) = d([y],x) = d(y',[x])$, so one belongs to
    ${\overline B}([x],r)$ if and only if the other does.
  \item[\impnext] As $F$ and each $[x]$ are compact and closed under
    $\sim$, we have from \fref{lem:DistFromCompact} that
    ${\overline B}(F,r) = \bigcup_{x \in F} B(x,r) = \bigcup_{x \in F} B([x],r)$.
  \item[\impnext] Immediate.
  \item[\impfirst]
    It is enough to show that $d(x,[y]) \geq d([x],y)$.
    Let $B$ be a family as in the assumption, and assume that
    $d(y,[x]) > r \in \bQ^+$.
    Then we can find a family $B_0 \subseteq B$ closed under finite
    intersections such that $\bigcap B_0 = [x]$.
    As ${\overline B}(y,r) \cap [x] = \emptyset$, we obtain by compactness that
    ${\overline B}(y,r) \cap F = \emptyset$ for some $F \in B_0$.
    Then $y \notin {\overline B}(F,r)$,
    and since the latter is closed under $\sim$:
    $[y] \cap {\overline B}(F,r) = \emptyset$.
    In particular $d([y],x) > r$.
  \end{cycprf}
\end{proof}

Let $\fA = C(X,[0,1])$.
It is known that the closed equivalence relations on $X$ are precisely
those of the form $\sim_\fB$ where
$\fB \subseteq \fA$ and $x \sim_\fB y \Longleftrightarrow (\forall f\in\fB)(f(x) = f(y))$.
In this case for each $f \in \fB$ there is a unique
$f' \in C(Y,[0,1])$ such that $f = f' \circ \pi$.
If $\fB$ is closed under $\lnot$, $\half$ and $\dotminus$, then the set
$\fB' = \{f'\colon f \in \fB\}$ dense in uniform convergence in $C(Y,[0,1])$.

\begin{lem}
  \label{lem:PrecQuotRel3}
  Let $\fB \subseteq \fA$, and assume $\fB$ is closed under
  $\lnot$, $\half$ and $\dotminus$.
  Then the following are equivalent:
  \begin{enumerate}
  \item For all $x,y \in X$: $d(x,[y]) = d([x],y)$.
  \item For all $r \in \bQ^+$, $f \in \fB$:
    ${\overline B}(f^{-1}(\{0\},r))$ is closed under $\sim_\fB$.
  \end{enumerate}
\end{lem}
\begin{proof}
  \begin{cycprf}
  \item[\impnext]
    Follows from \fref{lem:PrecQuotRel2}, as every set of the form
    $f^{-1}(\{0\})$ is closed in $X$ and closed under $\sim_\fB$.
  \item[\impfirst]
    Let $B = \{f^{-1}(\{0\})\colon f \in \fB\}$.
    For every $f,g \in \fB$ we also have $f\vee g \in \fB$, so $B$ is closed
    under finite intersections.
    If $x \not\sim_\fB y$, there is $f \in \fB$ such that
    $f(x) \neq f(y)$.
    Possibly replacing $f$ with $\lnot f \in \fB$ we have $f(x) < f(y)$.
    If $f(x) < s < f(y)$ and $s \in [0,1]$ is dyadic, then
    we can replace $f$ with $f \dotminus s \in \fB$, and have
    $f(x) = 0 < f(y)$.
    Thus $B$ satisfies the hypotheses of
    \fref{lem:PrecQuotRel2}\fref{item:PrecQuotRel2Family}.
  \end{cycprf}
\end{proof}

As we use continuous functions to prescribe equivalence relations, we
need to measure how many such functions are required to capture a
piece of information.

\begin{dfn}
  \label{dfn:DefComplex}
  \begin{enumerate}
  \item
    The \emph{definition complexity} $\defcomp(K)$
    of a  closed subset $K \subseteq X$ is the smallest
    infinite cardinal $\kappa$ such that there exists a family of
    continuous functions
    $\{f_i\colon i<\kappa\} \subseteq \fA$ which define $K$ in the
    sense that $K = \bigcap_i f_i^{-1}(\{0\})$.
  \item The definition complexity
    of the metric $d$ is defined as:
    \begin{gather*}
      \defcomp(d) =
      \sup\{\defcomp({\overline B}(K,r))\colon r>0, \defcomp(K) = \aleph_0\}.
    \end{gather*}
  \end{enumerate}
\end{dfn}

Note that if $\defcomp(K) = \aleph_0$ then $K$ is the zero set
of a single continuous function.
In case $\defcomp(K) = \kappa > \aleph_0$ we can find $\{K_i\colon i < \kappa\}$
such that $\defcomp(K_i) = \aleph_0$, $K = \bigcap_i K_i$, and the family
$\{K_i\colon i < \kappa\}$ is closed under finite intersections.
It follows by compactness that
${\overline B}(K,r) = \bigcap_i {\overline B}(K_i,r)$.
We conclude that
$\defcomp({\overline B}(K,r)) \leq \defcomp(K) +
\defcomp(d)$.

\begin{thm}
  \label{thm:PrecQuot}
  Let $X$ be a compact topometric space,
  $\fB \subseteq \fA = C(X,[0,1])$.
  Then there exists
  $\fB \subseteq \fB' \subseteq \fA$ such that:
  \begin{enumerate}
  \item $|\fB'| \leq |\fB| + \defcomp(d)$.
  \item $Y = X/{\sim_{\fB'}}$ admits a topometric
    structure for which the projection $\pi\colon X\to Y$ is a precise topometric
    quotient map.
  \end{enumerate}
  Moreover, this quotient structure is unique given by the quotient
  topology and the metric:
  \begin{gather*}
    d_Y([x],[y]) = d_X([x],[y]) = \inf \{d_X(x',y')\colon x'\sim_{\fB'}x, y'\sim_{\fB'}y\}.
  \end{gather*}
\end{thm}
\begin{proof}
  For each $f \in \fA$ and $r \in \bQ^+$ choose a family
  $\fB_{f,r} \subseteq \fA$ such that $|\fB_{f,r}| \leq \defcomp(d)$ and
  ${\overline B}(f^{-1}(\{0\}),r) = \bigcap \{g^{-1}(\{0\})\colon g \in \fB_{f,r}\}$.

  Let $\fB_0$ be the closure of $\fB$ under $\lnot$, $\half$ and
  $\dotminus$, so $|\fB_0|\leq |\fB| + \aleph_0$.
  Given $\fB_n \subseteq \fA$ closed under these connectives, let
  $\fB_{n+1}$ be the closure under the connectives of
  $\fB_n\cup \bigcup_{f \in \fB_n,r \in \bQ^+} \fB_{f,r}$.

  Let $\fB' = \bigcup_n \fB_n$.
  Then $\fB \subseteq \fB' \subseteq \fA$, $|\fB'| \leq |\fB| + \defcomp(d)$, and
  $\fB'$ satisfies the hypotheses
  of \fref{lem:PrecQuotRel3}.
  We conclude using \fref{lem:PrecQuotRel}.
\end{proof}

\begin{cor}
  \label{cor:PrecQuot}
  Let $X$ be a compact topometric space,
  $\fB \subseteq \fA = C(X,[0,1])$.
  Then $X$ admits a precise topometric quotient
  $\pi\colon X \to Y$ such that each member of $\fB$ factors through $\pi$ and
  $\wt(Y) \leq |\fB| + \defcomp(d)$.
\end{cor}
\begin{proof}
  Immediate from \fref{thm:PrecQuot}, since
  $\wt(X/{\sim_{\fB'}}) \leq |\fB'| + \aleph_0 \leq |\fB| + \defcomp(d)$.
\end{proof}

\begin{dfn}
  \label{dfn:EnoghQuot}
  Let $X$ be a compact topometric space.
  \begin{enumerate}
  \item A family $\cQ$ of (isomorphism classes of) quotients of $X$ is
    \emph{sufficient} if for every subset $\fB \subseteq C(X,[0,1])$ there is
    a quotient $(Y,\pi) \in \cQ$ (where $\pi\colon X\to Y$ is the quotient map)
    such that $\wt(Y) \leq |\fB| + \aleph_0$ and every member of $\fB$ factors
    via $\pi$.
  \item We say that $X$ has \emph{enough quotients} if $X$ admits a
    sufficient family of quotients.
  \item We say that $X$ has \emph{enough precise quotients} if $X$
    admits a sufficient family of precise quotients.
  \end{enumerate}
\end{dfn}

\begin{thm}
  \label{thm:EnoghQuot}
  Let $X$ be a topometric space.
  Then the following are equivalent:
  \begin{enumerate}
  \item $\defcomp(d_X) = \aleph_0$.
  \item The family of all precise quotients of $X$ is sufficient.
  \item $X$ has enough precise quotients.
  \end{enumerate}
\end{thm}
\begin{proof}
  \begin{cycprf}
  \item[\impnext] By \fref{cor:PrecQuot}.
  \item[\impnext] Immediate.
  \item[\impfirst] Let $K \subseteq X$ be a zero set, $K' = {\overline B}(K,r)$,
    and we need to show that $\defcomp(K') = \aleph_0$ as well.
    Say that $K = f^{-1}(0)$, and let $\pi\colon X\to Y$ be a precise quotient
    such that $\wt(Y) = \aleph_0$ and $f = f' \circ \pi$.
    Let $\bar K = {f'}^{-1}(0)$, $\bar K' = {\overline B}(\bar K,r)$.
    Then $K = \pi^{-1}(\bar K)$, so $K' = \pi^{-1}(\bar K')$ by
    preciseness.
    On the other hand, as $\wt(Y) = \aleph_0$ every closed set is a zero
    set, so say $\bar K' = g^{-1}(0)$.
    Then $K' = (g\circ\pi)^{-1}(0)$, as desired.
  \end{cycprf}
\end{proof}

The topometric spaces we are interested in are type spaces,
with either the standard metric or some other (e.g., perturbation)
metric.
Such spaces almost always have enough precise quotients.
\begin{prp}
  \label{prp:CtblDefComp}
  \begin{enumerate}
  \item Let $T$ be a theory in a language of arbitrary size.
    Let $d'$ be a metric on $\tS_n(T)$,
    and assume that for every $r$ the set of $2n$-tuples
    $\{(\bar a,\bar b)\colon d'(\tp(\bar a),\tp(\bar b)) \leq r\}$
    is type-definable using only countably many symbols
    from the language.
    Then $(\tS_n(T),d')$ has enough precise quotients.
  \item Let $T$ be a theory in a language of arbitrary size.
    Then $(\tS_n(T),d)$ has enough precise quotients where $T$ is the
    standard metric.
    (Since we may name parameters in the language, this also applies
    to $\tS_n(A)$ for any set of parameters $A$.)
  \item Let $M$ be an $\aleph_1$-saturated and strongly $\aleph_1$-homogeneous
    structure in a countable language, and let $d'$ be
    a lower semi-continuous metric on $\tS_n(M)$ invariant under
    the action of $\Aut(M)$.
    Then $(\tS_n(M),d')$ has enough precise quotients.
  \item Let $T$ be a theory in a language of arbitrary size,
    $\varphi(\bar x,\bar y)$ a formula, $M$ a model.
    Then $(\tS_\varphi(M),d_\varphi)$ has enough precise quotients where $d_\varphi$ is
    the standard metric on $\tS_\varphi(M)$.
  \end{enumerate}
\end{prp}
\begin{proof}
  For the first item, let $K \subseteq \tS_n(T)$ be a zero set.
  Then $K$ can be defined using countably many symbols from the
  language, so ${\overline B}(K,r)$ can also be defined using countably many
  symbols and is therefore a zero set as well.
  It follows that $\defcomp(d') = \aleph_0$, and conclude
  using \fref{thm:EnoghQuot}.
  The second and third items are special cases of the first one.

  For the last item, even though we allow a language of arbitrary size
  we may replace it with a countable sub-language containing all
  symbols appearing in $\varphi$.
  By downward L\"owenheim-Skolem,
  the family of all spaces $\tS_\varphi(M')$ where
  $M' \preceq M$ is a sufficient family of precise quotients.
\end{proof}

\section{$d$-isolation}
\label{sec:Isol}

Usually we do not expect to find topologically isolated points in a
non-maximal topometric space, since by definition the topology cannot be
stronger than the metric.
Instead, we define a notion of isolation relative to the metric by
requiring the topology to be as strong as possible around a point,
i.e., to coincide with the metric:
\begin{dfn}
  \begin{enumerate}
  \item A point $x \in X$ is \emph{$d$-isolated} if the topology and the
    metric agree near $x$, i.e., if $x \in B(x,r)^\circ$ for all $r > 0$.
  \item It is \emph{weakly $d$-isolated} if we only have
    $B(x,r)^\circ \neq \emptyset$ for all $r > 0$.
  \end{enumerate}
\end{dfn}

The need for two notions of isolation may be bothering.
Indeed, in the case of the standard metric on type spaces
$d$-isolation and weak $d$-isolation are equivalent
(see \cite[Fact~1.8]{BenYaacov-Usvyatsov:dFiniteness}).
On the other hand, the distinction is unavoidable in some cases, e.g.,
that of perturbation metrics, and each notion plays its own role.

For example, we have:
\begin{lem}
  \label{lem:WIsolClsd}
  The set of weakly $d$-isolated points in a topometric space $X$ is
  metrically closed.
\end{lem}
\begin{proof}
  Let $x \in X$, and assume that $B(x,r)$ contains a weakly $d$-isolated
  point for all $r > 0$.
  This means that $B(x,r/2)$ contains a weakly $d$-isolated point
  $x_r$, and $B(x_r,r/2)^\circ \neq \emptyset$.
  Thus $B(x,r)^\circ \neq \emptyset$ for all $r > 0$, and $x$ is weakly $d$-isolated
  as well.
\end{proof}
In particular, a metric limit of $d$-isolated points is weakly
$d$-isolated, and we do not know in general that it is $d$-isolated.

One can push the notion of weak $d$-isolation a bit further, allowing
us to improve the previous observation a little.
Since the usefulness of this exercise is not clear we do it
briefly.
Define the \emph{weak $d$-isolation rank} of a point as follows:
the rank of every point is at least zero, and the rank of $x$ is at
least $\alpha + 1$ if $B(x,r)^\circ$ contains a point of rank $\alpha$ for every
$r > 0$.
Thus a point is weakly $d$-isolated if and only if it has weak
$d$-isolation rank one or more, and every truly $d$-isolated point has
rank $\infty$.
\begin{lem}
  \label{lem:WIsolRank}
  \begin{enumerate}
  \item The set of all points of weak $d$-isolation rank $\geq \alpha$ is
    metrically closed.
  \item In a complete topometric space
    the set of point of weak $d$-isolation rank $\infty$ is
    precisely the metric closure of the set of
    $d$-isolated points.
  \end{enumerate}
\end{lem}
\begin{proof}
  The first item is proved like \fref{lem:WIsolClsd}, and
  yields that every limit of $d$-isolated points has weak
  $d$-isolation rank $\infty$.
  Conversely, assume $x$ has weak $d$-isolation rank $\infty$.
  Fix $r > 0$, and let $x_0 = x$, $r_0 = r$.
  Given $x_n$ of weak $d$-isolation rank $\infty$ and $r_n > 0$
  the set $B(x_n,r_n)^\circ$ contains at least one point $x_{n+1}$ of weak
  $d$-isolation rank $\infty$ (as else the set of all ranks of points there
  is bounded by some ordinal).
  As the metric refines the topology $B(x_n,r_n)^\circ$ is
  metrically open, is there is $r_{n+1} > 0$ such that
  ${\overline B}(x_{n+1},r_{n+1}) \subseteq B(x_n,r_n)^\circ$,
  and we may further assume that $r_{n+1} < r_n/2$.
  The sequence $(x_n\colon n > \omega)$ is Cauchy and thus converges to some
  point $y$.
  We observe that by construction $x_m \in {\overline B}(x_n,r_n)$ for all
  $m < n$ whereby
  $y \in {\overline B}(x_{n+1},r_{n+1}) \subseteq B(x_n,r_n)^\circ$.
  In particular $d(x,y) < r$.
  More generally, for every $\varepsilon < 0$ there is $n$ such that
  $r_n < \varepsilon/2$ in which case
  $y \in B(x_n,r_n)^\circ \subseteq B(y,\varepsilon)^\circ$
  and $y$ is $d$-isolated.
\end{proof}

On the other hand, weak $d$-isolation does not seem to pass to
sub-spaces while full $d$-isolation does.

\begin{lem}
  \label{lem:IsolSubSpace}
  Let $X \subseteq Y$ be topometric spaces, where $X$ carries the induced
  structure from $Y$, and let $x \in X$ be $d$-isolated in $Y$.
  Then it is $d$-isolated in $X$.
\end{lem}
\begin{proof}
  Let $r > 0$ and $U = B_Y(x,r)^\circ$.
  Then $x \in U\cap X$ by assumption, $U\cap X$ is open in $X$ and
  $U\cap X \subseteq B_X(x,r)$.
\end{proof}

Recall that a mapping of topological spaces $f\colon X\to Y$ is \emph{open}
if the image of every open set is open.
Say it is \emph{weakly open} if the image of every non-empty open set
has non-empty interior.
Then we have the following:
\begin{thm}
  \label{thm:IsolTower}
  Let $\pi\colon X \to Y$ be a morphism between locally compact
  topometric spaces.
  Let $x \in X$, $y = \pi(x) \in Y$, and $Z = \pi^{-1}(y)$ the fibre over $y$
  with the induced topometric structure.
  Then:
  \begin{enumerate}
  \item \label{item:IsolTowerDownFibre}
    If $x$ is $d$-isolated in $X$ then it is $d$-isolated in $Z$.
  \item \label{item:IsolTowerDownImage}
    If $x$ is (weakly) $d$-isolated in $X$
    and $\pi$ is (weakly) open then $y$ is (weakly) $d$-isolated in
    $Y$.
  \item  \label{item:IsolTowerUp}
    If $\pi$ is an epimorphism, $x$ is (weakly) $d$-isolated in $Z$
    and $y$ (weakly) $d$-isolated in $Y$ then $x$ is (weakly)
    $d$-isolated in $X$.
  \end{enumerate}
\end{thm}
\begin{proof}
  There is no harm in replacing $X$ with a compact neighbourhood of
  $x$ and replacing $Y$ with its image, so we may assume all spaces
  are compact, and thus that $\pi$ is uniformly continuous.
  Let $\Delta\colon (0,\infty) \to (0,\infty)$ be a uniform continuity modulus for $\pi$,
  meaning that for all $\varepsilon > 0$, if $d_X(x',x'') < \Delta(\varepsilon)$
  then $d_Y(\pi(x'),\pi(x'')) \leq \varepsilon$.

  For \fref{item:IsolTowerDownFibre} just use \fref{lem:IsolSubSpace}.

  For \fref{item:IsolTowerDownImage}
  assume first that $x$ is weakly $d$-isolated and $\pi$ weakly open.
  Fix $r > 0$.
  Let $U = B_X(x,\Delta(r))^\circ$ in $X$.
  Then $U \neq \emptyset$, whereby $\emptyset \neq \pi(U)^\circ \subseteq B_Y(y,r)^\circ$ in $Y$.
  If $x$ is $d$-isolated and $\pi$ open we have furthermore that
  $x \in U$ and $y \in \pi(U) = \pi(U)^\circ$.

  We now prove \fref{item:IsolTowerUp}.
  As $\pi$ is assumed to be an epimorphism we may further assume that
  if $d_Y(y',\pi(x')) < \Delta(\varepsilon)$ then $d_X(\pi^{-1}(y'),x') \leq \varepsilon$.
  Assume that $x$ is weakly $d$-isolated in $Z$ and $y$ is
  weakly $d$-isolated in $Y$.
  Fix $r > 0$.
  Let $U = B_Z(x,r/2)^\circ$ in $Z$, so $U \neq \emptyset$ and we can fix some
  $z \in U$.
  Then there is an open set $V \subseteq X$ such that $U = Z\cap V$.
  We can find an open subset $V' \subseteq X$ such that
  $z \in V' \subseteq \bar V' \subseteq V$.
  Then $\bar V' = \bigcap_{s > 0} {\overline B}(\bar V',s)$, and as
  $\bar V'$ has closed metric
  neighbourhoods and $X$ is compact there is $s > 0$ such that
  ${\overline B}_X(\bar V',s) \subseteq V$.
  Decreasing $s$ further we may assume that $s \leq r/2$ and
  ${\overline B}_X(z,s) \subseteq V'$.

  Let $W = B_Y(y,\Delta(s))^\circ$, and $R = \pi^{-1}(W)\cap V'$, so $R$ is open.
  Then $W \neq \emptyset$, and given $y' \in W$ we know that
  $d_X(z,\pi^{-1}(y')) \leq s$.
  As $\pi^{-1}(y')$ is compact, we obtain by \fref{lem:DistFromCompact} that
  $\emptyset \neq \pi^{-1}(y')\cap {\overline B}_X(z,s) \subseteq R$.
  Finally we show that $R \subseteq B_X(x,r)$.
  Indeed, let $x' \in R$.
  Then $\pi(x') \in W$, so
  $d_Y(\pi(x'),y) < \Delta(s) \Longrightarrow d(x',Z) \leq s$, and by
  \fref{lem:DistFromCompact}
  there is $x'' \in Z$ such that $d_X(x',x'') \leq s$.
  Then $x' \in V' \Longrightarrow x'' \in {\overline B}(V',s) \subseteq V$, so
  $x'' \in Z\cap V = U \subseteq B_Z(x,r/2)$.
  Thus $x' \in B_X(x,r)$, as desired.
  We conclude that $B_X(x,r)^\circ \supseteq R \neq \emptyset$, so
  $x$ is weakly $d$-isolated.

  In case $x$ is $d$-isolated in $Z$ and $y$ in $Y$ then we may choose
  $z = x \in V'$ and we  know that $y \in W$ whereby $x \in R$, so the same
  argument shows that $x$ is $d$-isolated in $X$.
\end{proof}

\begin{cor}
  Let $\pi\colon X \to Y$ be an open epimorphism of locally
  compact topometric spaces.
  Let $x \in X$, $y = \pi(x) \in Y$, and $Z = \pi^{-1}(y)$ the fibre over $y$
  with the induced topometric structure.

  Then $x$ is $d$-isolated in $X$ if and only if $x$ is $d$-isolated
  in $Z$ and $y$ is $d$-isolated in $Y$.
\end{cor}

\begin{cor}
  \label{cor:IsolTower}
  Let $\pi\colon X \to Y$ be an epimorphism of locally
  compact topometric spaces.
  Let $x \in X$, $y = \pi(x) \in Y$, and $Z = \pi^{-1}(y)$ the fibre over $y$
  with the induced topometric structure.
  Let $(Z,d_Z)$ be another topometric structure with the same
  underlying topological space $Z$, where $d_Z$ is finer than $d_X$, so
  $\id\colon (Z,d_Z) \to (Z,d_X)$ is a morphism.

  If $x$ is $d_Z$-isolated in $Z$ and $y$ is $d_Y$-isolated in $Y$
  then $x$ is $d_X$-isolated in $X$.
\end{cor}
\begin{proof}
  Since $\id\colon (Z,d_Z) \to (Z,d_X)$ is open we have that
  $x$ is $d_X$-isolated in $Z$ by
  \fref{thm:IsolTower}\fref{item:IsolTowerDownImage}.
  Now apply \fref{thm:IsolTower}\fref{item:IsolTowerUp}.
\end{proof}

\section{Cantor-Bendixson ranks}
\label{sec:CB}

In classical topological spaces the Cantor-Bendixson derivative
consists of removing isolated points.
One crucial property of the derivative is that it is a closed
subspace.
In the topometric setting the situation is more complicated.
If we simply tried to take out the (weakly) $d$-isolated points the
derivative would no longer be closed, and the machinery would break
down.
We resolve this difficulty by viewing the classical Cantor-Bendixson
derivative as consisting of removing open sets which are
``small'' (singletons, or finite sets).
In a topometric space the metric gives rise to notions of smallness
which allow to recover much of the classical theory concerning
Cantor-Bendixson analysis.

Similar extensions of the classical Cantor-Bendixson analysis were
also defined and used by
Newelski \cite{Newelski:DiameterLascarStrongType}.

\subsection{General definitions}

Fix a topometric space $(X,\sT,d)$.
We consider several natural notions of smallness
which depend on a parameter $\varepsilon > 0$,
none of which \emph{a priori}
better than another:

\begin{dfn}
  Let $A \subseteq X$, $\varepsilon > 0$, $\alpha \leq \omega$.
  We say that $A$ is \emph{$\varepsilon$-$\alpha$-finite} if there
  is no subset $\{a_i\colon i \leq \alpha\} \subseteq A$ satisfying
  $d(a_i,a_j) > \varepsilon$ for all $i<j\leq\alpha$.
  \begin{enumerate}
  \item We observe
    that $A$ is $\varepsilon$-$1$-finite if and only if
    $\diam(A) \leq \varepsilon$.
  \item We say that $A$ is \emph{$\varepsilon$-finite} if
    it is $\varepsilon$-$n$-finite for some $n < \omega$.
  \item We say that $A$ is \emph{$\varepsilon$-bounded} if
    it is $\varepsilon$-$\omega$-finite.
  \end{enumerate}
\end{dfn}

These notions of smallness relate as follows:
\begin{lem}
  \label{lem:SmallOpenSets}
  Let $U \subseteq X$ be open, $\varepsilon > 0$.
  \begin{enumerate}
  \item If $\diam(U) \leq \varepsilon$ then $U$ is $\varepsilon$-finite.
  \item If $U$ is $\varepsilon$-finite then it is $\varepsilon$-bounded.
  \item If $U$ is $\varepsilon$-bounded then it is the union of
    open sets of diameter $\leq 2\varepsilon$.
  \end{enumerate}
\end{lem}
\begin{proof}
  All but the last property are obvious.
  So assume $U$ is $\varepsilon$-bounded.
  Let $a_0 \in U$ and construct by induction a maximal set
  $\{a_i\colon i <\alpha\} \subseteq U$ satisfying $d(a_i,a_j) > \varepsilon$
  for all $i < j < \alpha$.
  Since $U$ is $\varepsilon$-bounded $\alpha$ must by finite.
  Let $V = U \setminus \bigcup_{0<i<\alpha} {\overline B}(a_i,\varepsilon)$.
  Then $V$ is open, since each ${\overline B}(a_i,\varepsilon)$ is closed,
  $a_0 \in V \subseteq U$, and $\diam(V) \leq 2\varepsilon$ by the maximality of the set
  $\{a_i\colon i < \alpha\}$.
\end{proof}

To each notion of smallness we associate a Cantor-Bendixson
derivative and rank.
We can also define Cantor-Bendixson derivative based on
(weakly) $d$-isolation, but showing these have the desired properties
is trickier.
\begin{dfn}
  \label{dfn:CBrank}
  Let $\varepsilon > 0$ and $* \in \{d,f,b,i,wi\}$
  \begin{enumerate}
  \item We define the \emph{$(*,\varepsilon)$-Cantor-Bendixson derivative} of
    $X$ as:
    \begin{gather*}
      \begin{array}{ll}
        X'_{d,\varepsilon} & = X \setminus \bigcup\{U \subseteq X\colon U \text{ open, } \diam(U) \leq \varepsilon\} \\
        X'_{f,\varepsilon} & = X \setminus \bigcup\{U \subseteq X\colon U \text{ open and $\varepsilon$-finite}\} \\
        X'_{b,\varepsilon} & = X \setminus \bigcup\{U \subseteq X\colon U \text{ open and $\varepsilon$-bounded}\} \\
        X'_{i,\varepsilon} & = X \setminus \bigcup\{{\overline B}(a,\varepsilon)^\circ\colon a \in X \text{ $d$-isolated}\} \\
        X'_{wi,\varepsilon} & = X \setminus \bigcup\{{\overline B}(a,\varepsilon)^\circ\colon a \in X \text{ weakly $d$-isolated}\}.
      \end{array}
    \end{gather*}
  \item We define the \emph{$(*,\varepsilon)$-Cantor-Bendixson derivative
      sequence}:
    \begin{gather*}
      \begin{array}{llr}
        X^{(0)}_{*,\varepsilon} & = X, \\
        X^{(\alpha+1)}_{*,\varepsilon} & = \bigl( X^{(\alpha)}_{*,\varepsilon} \bigr)'_{*,\varepsilon}, \\
        X^{(\alpha)}_{*,\varepsilon} & = \bigcap_{\beta<\alpha} X^{(\beta)}_{*,\varepsilon}, & (\alpha \text{ limit}) \\
        X^{(\infty)}_{*,\varepsilon} & = \bigcap_\alpha X^{(\alpha)}_{*,\varepsilon} = X^{|X|^+}_{*,\varepsilon}.
      \end{array}
    \end{gather*}
  \item 
    For $A \subseteq X$ we define its \emph{$(*,\varepsilon)$-Cantor-Bendixson rank}:
    \begin{gather*}
      \CB_{*,\varepsilon}(A) = \sup \{\alpha\colon X^{(\alpha)}_{*,\varepsilon} \cap A \neq \emptyset\}.
    \end{gather*}
    Note that if $A$ is compact then supremum is attained as a
    maximum.
  \item
    We say that $X$ is \emph{$(*,\varepsilon)$-CB-analysable} if
    $\CB_{*,\varepsilon}(X) < \infty$, i.e., if $X^{(\infty)}_{*,\varepsilon} = \emptyset$.
    It is \emph{$*$-CB-analysable} if it is $(*,\varepsilon)$-CB-analysable for
    all $\varepsilon > 0$.
  \end{enumerate}
\end{dfn}

\begin{dfn}
  In case $K \subseteq X$ is compact and has a Cantor-Bendixson rank we may
  define its \emph{Cantor-Bendixson degree}, in a manner depending on
  the kind of rank in question:
  \begin{enumerate}
  \item If $\CB^X_{d,\varepsilon}(K) = \alpha < \infty$ then
    $X^{(\alpha)}_{d,\varepsilon}\cap K \subseteq \bigcup_{i<n} U_i$ where each $U_i \subseteq X^{(\alpha)}_{d,\varepsilon}$
    is open of diameter $\leq \varepsilon$, and we define
    $\CBd^X_{d,\varepsilon}(K)$ to be the minimal such $n$.
  \item If $\CB^X_{f,\varepsilon}(K) = \alpha < \infty$ then
    $X^{(\alpha)}_{f,\varepsilon}\cap K \subseteq U$ where $U \subseteq X^{(\alpha)}_{f,\varepsilon}$
    is open and $\varepsilon$-$n$-finite for some $n$,
    and we define
    $\CBd^X_{f,\varepsilon}(K)$ to be the minimal such $n$.
  \item If $\CB^X_{b,\varepsilon}(K) = \alpha < \infty$ then
    $X^{(\alpha)}_{b,\varepsilon}\cap K \subseteq U$ where $U \subseteq X^{(\alpha)}_{b,\varepsilon}$
    is open and $\varepsilon$-bounded, and thus
    $2\varepsilon$-$n$-finite for some $n$,
    and we define
    $\CBd^X_{b,\varepsilon}(K)$ to be the minimal such $n$.
  \item The definition in the previous item is based on the one we
    have already given in \cite{BenYaacov:Morley} (see
    \fref{rmk:CBInLiterature} below).
    Alternatively, we observe that
    since $U$ is $\varepsilon$-bounded there exists a
    maximal finite subset
    $\{a_i\colon i<n\} \subseteq U$ verifying
    $d(a_i,a_j) > \varepsilon$ for all $i < j < n$
    and we may define $\CBd^X_{b',\varepsilon}(K)$ to be the smallest
    $n$ for which this is possible.
  \end{enumerate}
\end{dfn}

\begin{rmk}
  \label{rmk:CBInLiterature}
  \begin{enumerate}
  \item 
    In the case of a maximal topometric space, which is just
    the topometric representation of a classical topological space, these
    notions all coincide (for $\varepsilon$ small enough) with the classical
    Cantor-Bendixson ranks and derivatives.
  \item The Cantor-Bendixson rank which was defined in
    \cite{BenYaacov-Usvyatsov:CFO}
    is $\CB_{d,\varepsilon}$.
  \item Also, one can verify that the $\varepsilon$-Morley rank of a closed or
    open set $A$ as defined in \cite{BenYaacov:Morley} coincides with
    $\CB_{b,\varepsilon}(A)$ where we view $A$ as a subset of the space of types
    over a sufficiently saturated model, equipped with the standard
    metric.
    The $\varepsilon$-Morley degree defined there coincides with
    $\CBd_{b,\varepsilon}(A)$.
  \end{enumerate}
\end{rmk}

The three notions of Cantor-Bendixson rank based on notions of
smallness are tightly related, and in particular define a
unique notion of CB-analysability.
\begin{prp}
  \label{prp:CBdbf}
  For every topometric space $X$, ordinal $\alpha$ and $\varepsilon>0$:
  \begin{gather*}
    X^{(\alpha)}_{d,2\varepsilon}
    \subseteq X^{(\alpha)}_{b,\varepsilon}
    \subseteq X^{(\alpha)}_{f,\varepsilon}
    \subseteq X^{(\alpha)}_{d,\varepsilon}.
  \end{gather*}
  If follows for all $A \subseteq X$:
  \begin{gather*}
    \CB_{d,2\varepsilon}(A) \leq \CB_{b,\varepsilon}(A)
    \leq \CB_{f,\varepsilon}(A) \leq \CB_{d,\varepsilon}(A).
  \end{gather*}
  In particular, being $*$-CB-analysable for $* \in \{d,f,b\}$ are all
  equivalent properties, and from now on we shall refer to them as being
  \emph{CB-analysable}.
\end{prp}
\begin{proof}
  It follows from \fref{lem:SmallOpenSets} that
  if $X \subseteq Y \subseteq Z \subseteq W$ then
  $X'_{d,2\varepsilon} \subseteq Y'_{b,\varepsilon}
  \subseteq Z'_{f,\varepsilon} \subseteq W'_{d,\varepsilon}$, and from
  there proceed by induction.
\end{proof}

As we shall see below, the ranks $\CB_{f,\varepsilon}$ and
$\CB_{b,\varepsilon}$ are
somewhat easier to study than the $\CB_{d,\varepsilon}$.
At the same time, the degrees associated to $\CB_{d,\varepsilon}$ and
$\CB_{f,\varepsilon}$ are more elegant than those associated with
$\CB_{b,\varepsilon}$.
This ``comparative study'' suggests that among these three, the most
convenient rank to use is $\CB_{f,\varepsilon}$.

\begin{prp}
  \label{prp:DenseIsol}
  Assume that $X$ is locally compact and CB-analysable.
  Then the $d$-isolated points are dense.
\end{prp}
\begin{proof}
  Let $U \subseteq X$ be open, non-empty.
  As $X$ is locally compact, we may replace $U$ with a non-empty open
  subset such that $\bar U$ is compact.
  We construct a decreasing sequence of non-empty open sets
  $(U_i\colon i < \omega)$ and numbers
  $(\varepsilon_i >0\colon i < \omega)$ such that $U_0 = U$,
  $\diam(U_{i+1}) \leq 2^{-i}$ and $\bar U_{i+1} \subseteq U_i$.

  Start with $U_0 = U$.
  Given $U_i$ open, non-empty, let $V_i$ be open and non-empty such
  that $\bar V_i \subseteq U_i$.

  Let $x \in V_i$ be such that $\CB_{d,2^{-i}}(x) = \alpha$
  is minimal.
  This means that $V_i \subseteq X^{(\alpha)}_{d,2^{-i}}$ and that there is an
  open subset $W_{i+1}$ of $X^{(\alpha)}_{d,2^{-i}}$ such that $x \in W_{i+1}$ and
  $\diam(W_{i+1}) \leq 2^{-i}$.
  Let $U_{i+1} = V_i \cap W_{i+1}$.
  Since $V_i \subseteq X^{(\alpha)}_{d,2^{-i}}$, $U_{i+1}$ is open in $X$,
  $\diam(U_{i+1}) \leq 2^{-i}$ and $\bar U_{i+1} \subseteq \bar V_i \subseteq U_i$.

  In the end $\bigcap U_i = \bigcap \bar U_i$ is non-empty as a decreasing
  intersection on non-empty compact sets, and in fact consists
  of a single point $\{a\}$.
  It follows from the construction that
  for all $\varepsilon > 0$ the set $B(a,\varepsilon)$
  contains $U_i$ for some $i$, so $a \in U$ is $d$-isolated.
\end{proof}

\begin{cor}
  Let $X$ be a locally compact topometric space.
  Then the following are equivalent:
  \begin{enumerate}
  \item $X$ is CB-analysable.
  \item For all locally compact $\emptyset \neq Y \subseteq X$ the $d$-isolated points
    of the topometric space $Y$ (with the induced structure) are dense.
  \item For all closed $\emptyset \neq Y \subseteq X$ the topometric space $Y$ (with the
    induced structure) contains a weakly $d$-isolated point.
  \end{enumerate}
\end{cor}
\begin{proof}
  \begin{cycprf}
  \item[\impnext] By \fref{prp:DenseIsol}.
  \item[\impnext] Immediate.
  \item[\impfirst] Let $\varepsilon > 0$, and assume that $\alpha$ is such that
    $X^{(\alpha)}_{d,\varepsilon} \neq \emptyset$.
    Then by assumption there is a point $a \in X^{(\alpha)}_{d,\varepsilon}$ which is
    weakly $d$-isolated there.
    Then the set $B(a,\varepsilon/2)^\circ$ (as calculated inside $X^{(\alpha)}_{d,\varepsilon}$)
    is non-empty and of diameter $\leq \varepsilon$, so
    $X^{(\alpha+1)}_{d,\varepsilon} \subsetneq X^{(\alpha)}_{d,\varepsilon}$.
    Therefore $X^{(\infty)}_{d,\varepsilon} = \emptyset$.
  \end{cycprf}
\end{proof}

We also obtain a converse for \fref{lem:WIsolClsd}:
\begin{cor}
  \label{cor:IsolDense}
  Let $X$ be a locally compact CB-analysable topometric space.
  Then the set of weakly $d$-isolated points is the metric closure of
  the set of $d$-isolated points.
\end{cor}
\begin{proof}
  One inclusion follows from \fref{lem:WIsolClsd}.
  For the other, let $x$ be a weakly $d$-isolated point.
  Then for each $r > 0$ the set $B(x,r)^\circ$ is non-empty and by
  \fref{prp:DenseIsol} contains a $d$-isolated point.
\end{proof}

We can now show the relation with the Cantor-Bendixson ranks based on
(weakly) $d$-isolated points:
\begin{thm}
  Let $X$ be a locally compact topometric space.
  Then all notions of CB-analysability defined so far are equivalent.
  Moreover, if $X$ is CB-analysable then for every ordinal $\alpha$ and
  $\varepsilon>0$:
  \begin{gather*}
    X^{(\alpha)}_{d,2\varepsilon} \subseteq X^{(\alpha)}_{wi,\varepsilon} \subseteq X^{(\alpha)}_{i,\varepsilon} \subseteq X^{(\alpha)}_{d,\varepsilon} \\
    \intertext{Whereby for all $A \subseteq X$:}
    \CB_{d,2\varepsilon}(A) \leq \CB_{wi,\varepsilon}(A) \leq \CB_{i,\varepsilon}(A) \leq \CB_{d,\varepsilon}(A).
  \end{gather*}
\end{thm}
\begin{proof}
  Clearly if $X \subseteq Y \subseteq Z$ then
  $X'_{d,2\varepsilon} \subseteq Y'_{wi,\varepsilon} \subseteq Z'_{i,\varepsilon}$,
  whereby $X^{(\alpha)}_{d,2\varepsilon} \subseteq X^{(\alpha)}_{wi,\varepsilon} \subseteq X^{(\alpha)}_{i,\varepsilon}$ for all $\varepsilon$.
  Thus $i$-CB-analysable implies $wi$-CB-analysable implies
  CB-analysable.
  To close the circle assume that $X$ is CB-analysable.
  By \fref{prp:DenseIsol}, if $X^{(\alpha)}_{i,\varepsilon}$ is non-empty then it
  contains a $d$-isolated point, so
  $X^{(\alpha+1)}_{i,\varepsilon} \subsetneq X^{(\alpha)}_{i,\varepsilon}$.
  Thus $X^{(\infty)}_{i,\varepsilon} = \emptyset$.

  For the moreover part assume $X$ is CB-analysable.
  Assume $U \subseteq X$ is open and $\diam(U) \leq \varepsilon$.
  By \fref{prp:DenseIsol} $U$ contains a $d$-isolated point $a$, and
  clearly $U \subseteq {\overline B}(a,\varepsilon)^\circ$.
  It follows that
  $X'_{i,\varepsilon} \subseteq X'_{d,\varepsilon}$, and the rest follows.
\end{proof}

\begin{rmk}
  Clearly, if $\varepsilon > \delta$ then
  $X'_{*,\varepsilon} \subseteq X'_{*,\delta}$.
  One may therefore define
  \hbox{$X'_{*,\varepsilon^-}
    = \bigcap_{\delta < \varepsilon} X'_{*,\delta}$}
  and then proceed to define
  $X^{(\alpha)}_{*,\varepsilon^-}$ and
  $\CB^X_{*,\varepsilon^-}$ accordingly.
  We observe that for all $\varepsilon > \varepsilon'$:
  \begin{gather*}
    X^{(\alpha)}_{*,\varepsilon}
    \subseteq X^{(\alpha)}_{*,\varepsilon^-}
    \subseteq X^{(\alpha)}_{*,\varepsilon'}, \qquad
    \CB^X_{*,\varepsilon}(A)
    \leq \CB^X_{*,\varepsilon^-}(A)
    \leq \CB^X_{*,\varepsilon'}(A).
  \end{gather*}
  In particular, no new notion of CB-analysability arises with these
  ranks.
  For some specific model-theoretic considerations (e.g., local Shelah
  stability ranks) the $\CB_{*,\varepsilon^-}$ are more convenient to
  use than the $\CB_{*,\varepsilon}$ ranks we defined earlier.
  In the present paper we shall restrict our attention to
  ranks of the form $\CB_{*,\varepsilon}$.
\end{rmk}

\subsection{Cantor-Bendixson analysis of subspaces and quotients}

\begin{lem}
  \label{lem:CBSubspace}
  Let $X \subseteq Y$ be topometric spaces.
  Then for every $* \in \{d,b,f,i\}$, $x \in X$, $\varepsilon > 0$:
  $\CB_{*,\varepsilon}^X(x) \leq \CB_{*,\varepsilon}^Y(x)$.
  In particular if $Y$ is CB-analysable then so is $X$.
\end{lem}
\begin{proof}
  It is straightforward to verify that $X \subseteq Y$ implies
  $X'_{*,\varepsilon} \subseteq Y'_{*,\varepsilon}$
  for $* \in \{d,b,f,i\}$ (though not for $* = wi$,
  since a weakly $d$-isolated point is not necessarily so in a
  subspace).
  The statement follows.
\end{proof}

\begin{lem}
  \label{lem:CBImage}
  Let $\pi\colon X \to Y$ be a surjective morphism of compact topometric
  spaces, and let $\Delta$ be
  such that
  $d(x,y) \leq \Delta(\varepsilon)
  \Longrightarrow d(\pi(x),\pi(y)) \leq \varepsilon$.
  Then for all $b \in Y$ and $\varepsilon > 0$:
  $\CB^Y_{f,\varepsilon}(y) \leq \CB^X_{f,\Delta(\varepsilon)}(\pi^{-1}(y))$,
  i.e.,
  $Y^{(\alpha)}_{f,\varepsilon}
  \subseteq \pi\left(X^{(\alpha)}_{f,\Delta(\varepsilon)}\right)$
  for all $\alpha$.
  In particular, if $X$ is CB-analysable so is $Y$.
\end{lem}
\begin{proof}
  We show that
  $Y^{(\alpha)}_{f,\varepsilon}
  \subseteq \pi\left(X^{(\alpha)}_{f,\Delta(\varepsilon)}\right)$
  by induction on $\alpha$.
  For $\alpha = 0$ and limit this is clear from the induction hypothesis.
  For $\alpha+1$, let us assume that
  $y \in
  Y^{(\alpha+1)}_{f,\varepsilon}
  \setminus \pi\left(X^{(\alpha+1)}_{f,\Delta(\varepsilon)}\right)$.
  Then each $x \in \pi^{-1}(y)$ has an open neighbourhood $V_x$ such that
  $V_x \cap X^{(\alpha)}_{f,\Delta(\varepsilon)}$ is $\Delta(\varepsilon)$-finite.
  Since $\pi^{-1}(y)$ is compact it can be covered by a finite sub-family:
  $\pi^{-1}(y) \subseteq V_{x_0} \cup \ldots \cup V_{x_{k-1}} = V$.
  Then $V \cap X^{(\alpha)}_{f,\Delta(\varepsilon)}$ is
  $\Delta(\varepsilon)$-$n$-finite for some $n < \omega$.

  Let $N(y)$ denote the set of all open neighbourhoods of $y$.
  Then for all $U \in N(y)$ there are points
  $y_{U,i} \in U \cap Y^{(\alpha)}_{f,\varepsilon}$ for $i \leq n$ such that
  $i < j \Longrightarrow d(y_{U,i},y_{U,j}) > \varepsilon$.
  By the induction hypothesis there are
  $x_{U,i} \in \pi^{-1}(y_{U,i}) \cap X^{(\alpha)}_{f,\Delta(\varepsilon)}$.
  By assumption on $\pi,\Delta$:
  $i < j \Longrightarrow d(x_{U,i},x_{U,j}) > \Delta(\varepsilon)$.

  For each $i \leq n$, the net $(y_{U,i}\colon U \in N(y))$ converges
  to $y$.
  As $X$ is compact we can find a directed partially ordered set
  $(S,\leq)$ and a decreasing function $\sigma\colon S \to N(y)$ sending
  $s \mapsto U_s$ such that for each $i \leq n$ the  sub-net
  $(x_{U_s,i}\colon s \in S)$ converges in $X$, say to $z_i$.
  Then necessarily $\pi(z_i) = y$, i.e., $z_i \in \pi^{-1}(y)$.
  Thus for some $s \in S$ and $U = U_s \in N(y)$ we have
  $x_{U,i} \in V$ for all  $i \leq n$, in contradiction with
  $\Delta(\varepsilon)$-$n$-finiteness of $V \cap X^{(\alpha)}_{f,\Delta(\varepsilon)}$.
  This contradiction concludes the proof.
\end{proof}

\begin{dfn}
  Let $X$ be a topometric space and $\varepsilon > 0$.
  An \emph{$\varepsilon$-perfect tree} in $X$ is a tree of compact non-empty
  sets $\{F_\sigma\colon \sigma \in 2^{<\omega}\}$ where:
  \begin{enumerate}
  \item If $\sigma,\tau \in 2^{<\omega}$ and $\sigma < \tau$ (i.e., $\tau$ extends $\sigma$) then
    $F_\tau \subseteq F_\sigma$.
  \item For all $\sigma \in 2^{<\omega}$: $d(F_{\sigma0},F_{\sigma1}) > \varepsilon$.
  \end{enumerate}
\end{dfn}

\begin{lem}
  \label{lem:BasicPerfectTree}
  Let $X$ be a locally compact topometric space, $\varepsilon > 0$,
  and assume an $\varepsilon$-perfect tree
  $\{F_\sigma\colon \sigma \in 2^{<\omega}\}$ exists in $X$.
  Let $\sB$ be a base of closed sets for the topology on $X$, closed
  under finite intersections, and such that for every compact set
  $F \subseteq X$ there is a compact $F' \in \sB$ such that
  $F \subseteq F'$.
  Then there exists in $X$ an $\varepsilon$-perfect tree
  $\{F'_\sigma\colon \sigma < 2^{<\omega}\} \subseteq \sB$,
  such that moreover $F'_\sigma \supseteq F_\sigma$ for all
  $\sigma \in 2^{<\omega}$.
\end{lem}
\begin{proof}
  We let $F_\emptyset'$ be any compact member of $\sB$ containing $F_\emptyset$.

  We now proceed by induction on $|\sigma|$.
  Assume $F_\sigma \subseteq F_\sigma' \in \sB$ has been chosen.
  By assumption
  ${\overline B}(F_{\sigma0},\varepsilon)
  \cap F'_\sigma \cap F_{\sigma1} = \emptyset$.
  Since $F_{\sigma0}$ is compact,
  ${\overline B}(F_{\sigma0},\varepsilon)$ is closed and thus
  ${\overline B}(F_{\sigma0},\varepsilon) \cap F'_\sigma$ is compact.
  Since $F_{\sigma1}$ is an intersection of members of $\sB$, there is a
  finite sub-intersection $F'_{\sigma1}$ satisfying
  ${\overline B}(F_{\sigma0},\varepsilon)
  \cap F'_\sigma \cap F'_{\sigma1}
  = \emptyset$.
  Since $F'_\sigma \in \sB$ we may assume that
  $F'_\sigma \supseteq F'_{\sigma1}$, so $F'_{\sigma1}$ is
  compact and
  ${\overline B}(F_{\sigma0},\varepsilon) \cap F'_{\sigma1}
  = \emptyset$.
  Since $F_{\sigma0}$ is compact as well we get
  $d(F_{\sigma0},F'_{\sigma1}) > \varepsilon$,
  and since $\sB$ is closed under finite intersections we
  have $F'_{\sigma1} \in \sB$.
  We can now do the same thing to find $F'_{\sigma0} \in \sB$ compact such
  that $F_{\sigma0} \subseteq F'_{\sigma0} \subseteq F'_\sigma$ and
  $d(F'_{\sigma0},F'_{\sigma1}) > \varepsilon$.
  This completes the induction step.
\end{proof}

\begin{rmk}
  Notice that the second requirement on $\sB$ is not superfluous.
  Indeed, the set of all complements of bounded open sets in $\bR$ is a
  base for the closed sets and is closed under finite intersections,
  but equipping $\bR$ with the maximal (i.e., discrete) metric we obtain
  a locally compact topometric space in which
  \fref{lem:BasicPerfectTree} fails.

  On the other hand, let $X$ be a locally compact space and let $\sB$
  be the family of all zero sets of functions in $C(X,[0,1])$.
  Then it is easy to verify that $\sB$ satisfies all the assumptions
  of \fref{lem:BasicPerfectTree}.
\end{rmk}

\begin{prp}
  \label{prp:PerfectTree}
  Let $X$ be locally compact.
  Then $X$ is CB-analysable if and only if
  for no $\varepsilon > 0$ is there an $\varepsilon$-perfect tree in $X$.
\end{prp}
\begin{proof}
  Assume $X$ is not CB-analysable, say it is not
  $(d,\varepsilon)$-CB-analysable.
  We may replace $X$ with $X^{(\infty)}_{d,\varepsilon}$,
  so every non-empty open set
  has diameter greater than $\varepsilon$.
  Given an open set $\emptyset \neq U \subseteq X$ there are
  $x,y \in F$ be such that $d(x,y) > \varepsilon$.
  Then we can find compact neighbourhoods $F_0$ and $F_1$ of $x$ and
  $y$, respectively, such that $d(F_0,F_1) > \varepsilon$.
  In particular $F_0^\circ$ and $F_1^\circ$ are non-empty open sets.
  We can thus proceed by induction to construct the tree.

  Conversely, assume an $\varepsilon$-perfect tree $(F_\sigma\colon \sigma \in 2^{<\omega})$ exists.
  Let $F_n = \bigcup_{\sigma \in 2^n} F_\sigma$ and $F = \bigcap_n F_n$.
  Then $F \subseteq X$ is compact, and there is a natural surjective
  mapping $\pi\colon F \to 2^\omega$ sending $F_\tau = \bigcap_n F_{\tau\rest n}$ to $\tau$ for all
  $\tau \in 2^\omega$.
  Viewing $2^\omega$ with the natural topology and the discrete metric, $\pi$
  is a surjective morphism.
  Since $2^\omega$ is not CB-analysable, $F$ is not CB-analysable by
  \fref{lem:CBImage}, so $X$ is not CB-analysable by
  \fref{lem:CBSubspace}.
\end{proof}

\begin{thm}
  \label{thm:CBQuot}
  Let $X$ be a compact topometric space with enough precise quotients,
  and let $\cQ$ be a sufficient family of precise quotients of $X$
  (see \fref{dfn:EnoghQuot}).
  Then the following are equivalent:
  \begin{enumerate}
  \item $X$ is CB-analysable.
  \item All homomorphic images of $X$ are CB-analysable.
  \item All precise quotients of $X$ are CB-analysable.
  \item All $Y \in \cQ$ admitting a countable base are
    CB-analysable.
  \end{enumerate}
\end{thm}
\begin{proof}
  The first implication follows from \fref{lem:CBImage}.
  The second and third are immediate.

  For the last, assume $X$ is not CB-analysable.
  By \fref{prp:PerfectTree}, for some $\varepsilon > 0$
  there exists an $\varepsilon$-perfect tree
  $\{F_\sigma\colon \sigma \in 2^{<\omega}\}$ in $X$.
  Let $\sB$ be the collection of zero sets of continuous functions
  $f \in C(X,[0,1])$.
  Then $\sB$ satisfies the assumptions of \fref{lem:BasicPerfectTree},
  so we may assume that each $F_\sigma$ is the zero set of some
  $f_\sigma \in C(X,[0,1])$.
  Let $\fB = \{f_\sigma\colon \sigma \in 2^{<\omega}\}$.
  By definition there is a quotient $(Y,\pi) \in \cQ$
  such that $\wt(Y) = \aleph_0$ and
  each $f_\sigma$ factors as $f'_\sigma \circ \pi$ with $f'_\sigma \in C(Y,[0,1])$.
  Let $F'_\sigma = \pi(F_\sigma)$ for each $\sigma \in 2^{<\omega}$.
  Then as $f_\sigma$ factors through $\pi$ we have $F_\sigma = \pi^{-1}(F_\sigma')$, and
  since $\pi$ is precise we see that $d_Y(F'_{\sigma0},F'_{\sigma1}) > \varepsilon$ for all
  $\sigma$, so $\{F'_\sigma\colon \sigma \in 2^{<\omega}\}$ is an $\varepsilon$-perfect tree in $Y$, and $Y$
  is not CB-analysable.
\end{proof}

We now seek to relate CB-analysability of a topometric
space, its metric density character and its topological weight.

\begin{prp}
  \label{prp:WeightDenseChar}
  Let $X$ be a CB-analysable topometric space.
  Then $\|X\| \leq \wt(X)$.
\end{prp}
\begin{proof}
  We may assume $\wt(X) \geq \aleph_0$.

  Fix a base $\cB$ of open sets for $X$, $|\cB| = \wt(X)$.
  For each $U \in \cB$ and $n< \omega$, if $U$ contains a point of maximal
  $\CB_{d,2^{-n}}$-rank we let $x_{U,n}$ be such a point, otherwise
  $x_{U,n} \in U$ is an arbitrary point.
  Let $A = \{x_{U,n}\colon U \in \cB, n<\omega\}$.
  Clearly $|A| \leq \aleph_0\cdot|\cB| = \wt(X)$, and we claim $A$ is
  metrically dense.

  Indeed, let $x \in X$ and $\varepsilon>0$.
  Then for some $n$: $\varepsilon > 2^{-n}$.
  Let $\alpha = \CB_{d,2^{-n}}(x)$, and let $U_0 \subseteq X^{(\alpha)}_{d,2^{-n}}$ be
  relatively open of diameter $\leq 2^{-n}$.
  Let $U \subseteq X$ be open so that $U_0 = U\cap X^{(\alpha)}_{d,2^{-n}}$.
  Then $\CB_{d,2^{-n}}(U) = \alpha$, so $\CB_{d,2^{-n}}(x_{U,n}) = \alpha$, whereby
  $x_{U,n} \in U_0$.
  Thus $x_{U,n} \in B(x,\varepsilon)\cap A$.
\end{proof}

The converse does not hold in general (the disjoint union of a small
non-CB-analysable space with a large CB-analysable one would be a
counterexample).
The converse does hold when $\wt(X)$ is countable:
\begin{prp}
  \label{prp:CtblWeightDenseChar}
  Let $X$ be a locally compact topometric space with a countable base.
  Then $X$ is CB-analysable if and only if $\|X\| \leq \aleph_0$
  if and only if $\|X\| < 2^{\aleph_0}$.
\end{prp}
\begin{proof}
  If $X$ is CB-analysable then $\|X\| \leq \wt(X) = \aleph_0 < 2^{\aleph_0}$.
  Conversely, assume $X$ is not CB-analysable.
  Then for some $\varepsilon > 0$ there is an $\varepsilon$-perfect tree
  $\{F_\sigma\colon \sigma < 2^{<\omega}\}$ in $X$.
  By compactness, for each $\tau \in 2^\omega$ the intersection
  $F_\tau = \bigcap_{n<\omega} F_{\tau\rest_n}$ is non-empty, and we may choose
  $x_0 \in F_\tau$.
  Then $\tau \neq \tau' \Longrightarrow d(x_\tau,x_{\tau'}) > \varepsilon$, whereby $\|X\| \geq 2^{\aleph_0}$.
\end{proof}

In conjunction with \fref{thm:CBQuot} we obtain:
\begin{cor}
  \label{cor:CBQuotWeight}
  Let $X$ be a compact topometric space with enough precise quotients,
  and let $\cQ$ be a sufficient family of precise quotients of $X$.
  \begin{enumerate}
  \item $X$ is CB-analysable.
  \item Every homomorphic image $Y$ of $X$ satisfies
    $\|Y\| \leq \wt(Y)$.
  \item Every precise quotient $Y$ of $X$ satisfies
    $\|Y\| \leq \wt(Y)$.
  \item Every $Y \in \cQ$ with a countable base is
    metrically separable.
  \end{enumerate}
\end{cor}

\subsection{Comparing Cantor-Bendixson ranks of two spaces}

Earlier we compared the Cantor-Bendixson ranks of two topometric
spaces admitting a special relation, such as an inclusion or a
surjective homomorphism from one to the other.
Such (and other) homomorphisms can be identified with their graphs,
which are a special kind of closed relations between two spaces.
We shall now explore inequalities of Cantor-Bendixson ranks between
spaces admitting an arbitrary closed relation.

\begin{ntn}
  Let $X,Y$ be two compact spaces, $R \subseteq X \times Y$ a closed relation.
  For $x \in X$ and $A \subseteq Y$ we define:
  \begin{align*}
    R_x & = \{y \in Y\colon (x,y) \in R\}, \\
    R^{\forall A} & = \{x \in X\colon R_x \subseteq A\}, \\
    R^{\exists A} & = \{x \in X\colon R_x \cap A \neq \emptyset\}.
  \end{align*}
\end{ntn}

Note that:
\begin{enumerate}
\item For all $A \subseteq Y$:
  $R^{\forall(Y\setminus A)} = X \setminus R^{\exists A}$.
\item If $A \subseteq Y$ is closed then $R^{\exists A}$ is closed.
\item If $A \subseteq Y$ is open then $R^{\forall A}$ is open.
\end{enumerate}

For example, if $Y \subseteq X$ then
$R = \Delta_Y \subseteq Y \times X$ is a closed relation,
$R_y = \{y\}$, $R^{\forall A} = R^{\exists A} = A\cap Y$.
If $\pi\colon X \to Y$ is a projection, then
$R = \{(\pi(x),x)\colon x\in X\} \subseteq Y\times X$ is
again a closed relation, $R_y = \pi^{-1}(y)$, $R^{\exists A} = \pi(A)$.

The formalism of closed relations allows us to compare
Cantor-Bendixson ranks of spaces (or of subsets thereof) in very
general situations.
In particular, the following result, albeit more technical, is a
proper generalisation of \fref{lem:CBImage} and a partial
generalisation of \fref{lem:CBSubspace}.

Let us fix:
\begin{itemize}
\item A pair of locally compact topometric spaces $X,Y$.
\item A closed relation $R \subseteq X \times Y$ such that $R_x$ is compact for
  every $x \in X$.
\item  $\varepsilon,\delta > 0$ such that for all $(x,y),(x',y') \in R$:
  if $d_Y(y,y') \leq \delta$ then $d_X(x,x') \leq \varepsilon$.
\end{itemize}

\begin{lem}
  \label{lem:SmallRelPullback}
  If $U \subseteq Y$ is open and $\delta$-finite ($\delta$-bounded) then
  $R^{\forall U} \cap (R^{\exists Y})^\circ$ is open and
  $\varepsilon$-finite ($\varepsilon$-bounded).
\end{lem}
\begin{proof}
  Since both $R^{\forall U}$ and $(R^{\exists Y})^\circ$ are open so is their
  intersection.
  Assume now that for some ordinal $\alpha \leq \omega$ we have
  $x_i \in R^{\forall U} \cap (R^{\exists Y})^\circ$
  for $i \leq \alpha$ satisfying
  $i < j \leq \alpha \Longrightarrow d(x_i,x_j) > \varepsilon$.
  Then we can find $y_i \in R_{x_i} \subseteq U$ for all $i \leq \alpha$
  which necessarily satisfy
  $i < j \leq \alpha \Longrightarrow d(y_i,y_j) > \delta$.
\end{proof}

This remains true if we replace $\delta$-finite with
diameter $\leq \delta$ (i.e., $\delta$-$1$-finite).
However, it will not be of much use since the family of open
sets of diameter $\leq \delta$ is not closed under finite unions.

\begin{lem}
  \label{lem:RelCBAnalysis}
  Under the assumptions above, let
  $* \in \{f,b\}$, and set
  $X^{(\alpha)} = X^{(\alpha)}_{*,\varepsilon}$,
  $Y^{(\alpha)} = Y^{(\alpha)}_{*,\delta}$.
  Then for every ordinal $\alpha$:
  $(R^{\exists Y})^\circ \cap X^{(\alpha)} \subseteq R^{\exists Y^{(\alpha)}}$.
\end{lem}
\begin{proof}
  Indeed for $\alpha = 0$ this holds by assumption, and for $\alpha$ limit 
  a compactness argument shows that
  $R^{\exists Y^{(\alpha)}}
  = R^{\exists \bigcap_{\beta<\alpha} Y^{(\beta)}}
  = \bigcap_{\beta<\alpha} R^{\exists Y^{(\beta)}}$.

  Assume now that
  $(R^{\exists Y})^\circ \cap X^{(\alpha)}
  \subseteq R^{\exists Y^{(\alpha)}}$, and we shall
  prove this for $\alpha+1$.
  Let $x \in ((R^{\exists Y})^\circ\cap X^{(\alpha)})
  \setminus R^{\exists Y^{(\alpha+1)}}$: we need to show
  that $x \notin X^{(\alpha+1)}$.
  Let $K = R_x \cap Y^{(\alpha)} = (R^{(\alpha)})_x$,
  where $R^{(\alpha)} = R \cap (X^{(\alpha)} \times Y^{(\alpha)})$.

  Then $K$ is compact, and by assumption on $x$:
  $K \neq \emptyset$, $K \cap Y^{(\alpha+1)} = \emptyset$.
  The latter means that $K$ admits a covering
  $K \subseteq \bigcup_{i < \lambda} U_i$
  where each $U_i \subseteq Y^{(\alpha)}$ is relatively open and
  $\delta$-finite in case $* = f$ ($\delta$-bounded in case $* = b$).
  By compactness of $K$ we may take this union to be finite.
  Since a finite union of $\delta$-finite ($\delta$-bounded) sets is such,
  we find that $K \subseteq U$ where $U \subseteq Y^{(\alpha)}$
  is open and $\delta$-finite ($\delta$-bounded).
  Let
  $U' = (R^{(\alpha)})^{\forall U} \cap (R^{\exists Y})^\circ
  \subseteq X^{(\alpha)}$.
  Then $x \in U'$, and by \fref{lem:SmallRelPullback} applied
  to $R^{(\alpha)}$, $U'$ is open in $X^{(\alpha)}$ and $\varepsilon$-finite
  ($\varepsilon$-bounded).
  Thus $U' \cap X^{(\alpha+1)} = \emptyset$
  and $x \notin X^{(\alpha+1)}$, as desired.
\end{proof}

\begin{thm}
  \label{thm:RelCBRank}
  Let $X,Y$ be two locally compact topometric spaces.
  Let $R \subseteq X \times Y$ be a closed relation such that
  $R_x$ is compact for
  all $x \in X$.
  Let $\varepsilon,\delta > 0$ be such that for all
  $(x,y),(x',y') \in R$:
  if $d_Y(y,y') \leq \delta$ then $d_X(x,x') \leq \varepsilon$.
  Let $K \subseteq X$ be compact and $F \subseteq Y$ any set such that
  $K \subseteq (R^{\exists Y})^\circ \cap R^{\forall F}$.
  Then $\CB^X_{*,\varepsilon}(K) \leq \CB^Y_{*,\delta}(F)$ for $* \in \{f,b\}$.
\end{thm}
\begin{proof}
  Assume that
  $\CB^X_{*,\varepsilon}(K) \geq \alpha$, i.e., there exists
  $x \in K \cap X^{(\alpha)}_{*,\varepsilon}
  \subseteq (R^{\exists Y})^\circ \cap X^{(\alpha)}_{*,\varepsilon}
  \subseteq R^{\exists Y^{(\alpha)}_{*,\delta}}$,
  so there is $y \in R_x \cap Y^{(\alpha)}_{*,\delta}$.
  Since $K \subseteq R^{\forall F}$ we have $y \in R_x \subseteq F$.
  Thus $F \cap Y^{(\alpha)}_{*,\delta} \neq \emptyset$ and $\CB^Y_{*,\delta}(F) \geq \alpha$ as well.
\end{proof}

\begin{cor}
  \fref{lem:CBImage} can be obtained as a special case of
  \fref{thm:RelCBRank}.
  In fact it is enough to assume that $X$ and $Y$ are locally compact
  and $\pi\colon X\to Y$ is surjective with compact fibres.
\end{cor}
\begin{proof}
  We follow the notations and assumptions of \fref{lem:CBImage}.
  Let $R \subseteq Y \times X$ be the transposed graph of $\pi$, i.e.,
  $R = \{(\pi(x),x)\colon x \in X\}$.
  Then $R_y = \pi^{-1}(y)$ is compact for all $y \in Y$ by assumption.
  As $\pi$ is surjective, $R^{\exists X} = Y$.
  Let $\varepsilon > 0$ and $y \in Y$, and set $\delta = \Delta(\varepsilon)$,
  $K = \pi^{-1}(y) = R_y \subseteq X$.
  Then $y \in (R^{\exists X})^\circ \cap R^{\forall K}$, and by
  \fref{thm:RelCBRank}:
  $\CB^Y_{f,\varepsilon}(y) \leq \CB^X_{f,\delta}(\pi^{-1}(y))$.
\end{proof}
The same argument shows that 
$\CB^Y_{b,\varepsilon}(y) \leq \CB^X_{b,\delta}(\pi^{-1}(y))$ as well, improving
\fref{lem:CBImage}.

In particular, if $\pi\colon X\to Y$ is precise, then $\CB_{f,\varepsilon}$ ranks
go down, in the sense that
$\CB^Y_{f,\varepsilon}(K) \leq \CB^X_{f,\varepsilon}(\pi^{-1}(K))$
for $K \subseteq Y$.
We now seek sufficient conditions for equality.

\begin{dfn}
  Let $\pi\colon X\to Y$ be a precise surjective mapping
  of compact spaces.
  We say that $X$ is \emph{homogeneous over $Y$} (or more
  precisely, \emph{over $\pi$}) if for every
  $K \subseteq U \subseteq Y$, where $K$ is
  compact and $U$ open, and every
  countable set $X_0 \subseteq \pi^{-1}(K)$, there is an isometric automorphism
  $f$ of $X$ such that $\pi\circ f\rest_{X_0}$ is isometric with image in
  $U$.
\end{dfn}

(All the results below hold if we replace ``compact'' with ``locally
compact'' and require in addition that if $K \subseteq Y$ is compact then so
is $\pi^{-1}(K)$.)

\begin{prp}
  \label{prp:HomogTypeSpace}
  Let $M \preceq N$ be two structures, $M$ approximately
  $\aleph_0$-saturated and $N$ strongly $\aleph_1$-homogeneous.
  Then $\tS_n(N)$ is homogeneous over $\tS_n(M)$ and
  $\tS_\varphi(N)$ is homogeneous over $\tS_\varphi(M)$, each with the
  respective standard metric.
\end{prp}
\begin{proof}
  Let $K \subseteq U \subseteq \tS_n(M)$ be closed and open, respectively,
  and let $X_0 \subseteq \pi^{-1}(K) \subseteq \tS_n(N)$ be countable,
  say $X_0 = \{p_i\colon i < \omega\}$.
  Then one can find a formula $\varphi(\bar x,\bar b)$ with parameters
  $\bar b \in M$ such that
  $K \subseteq [\varphi(\bar x,\bar b) = 0] \subseteq [\varphi(\bar x,\bar b) \leq 1/2] \subseteq U$.
  Then there is some $\varepsilon > 0$ such that for all $\bar b' \in M$,
  if $d(\bar b,\bar b') < \varepsilon$ then
  $[\varphi(\bar x,\bar b') = 0] \subseteq U$ as well.

  For each $i < j < \omega$ one can find a countable set
  $A_{ij} \subseteq N$ such that
  $d(p_i\rest_{A_{ij}},p_j\rest_{A_{ij}}) = d(p_i,p_j)$.
  Then $A = \bigcup_{i<j< \omega} A_{ij}$ is countable.
  By approximate $\aleph_0$-saturation of $M$
  (see \cite[Fact~1.4]{BenYaacov-Usvyatsov:dFiniteness})
  we can find $\bar b'A' \subseteq M$
  such that
  $A\bar b \equiv A'\bar b'$ and $d(\bar b,\bar b') < \varepsilon$.
  By homogeneity of $N$ there exists $f \in \Aut(N)$ sending
  $A\bar b \mapsto A'\bar b'$.
  Then $f$ induces an isometric automorphism $\tilde f$
  of $\tS_n(N)$,
  $\pi \circ \tilde f(X_0) \subseteq [\varphi(\bar x,\bar b') = 0] \subseteq U$,
  and since $A' \subseteq M$ we get that $\pi \circ \tilde f$ is isometric on
  $X_0$.

  The proof for
  $\pi\colon \tS_\varphi(N) \to \tS_\varphi(M)$
  is essentially the same.
\end{proof}

\begin{lem}
  \label{lem:HomogIncCB}
  Let $\pi\colon X \to Y$ be precise,
  and assume $X$ is homogeneous over $Y$.
  Let $x \in X$ and $y = \pi(x) \in Y$.
  Then $\CB^Y_{*,\varepsilon}(y) \geq \CB^X_{*,\varepsilon}(x)$
  for all $\varepsilon > 0$, $* \in \{d,f,b\}$.
\end{lem}
\begin{proof}
  We only prove the case $* = f$, the others being similar.
  It is enough to show by induction on $\alpha$ that
  $\pi(X^{(\alpha)}_{f,\varepsilon}) \subseteq Y^{(\alpha)}_{f,\varepsilon}$.

  For $\alpha = 0$ or limit this is clear.
  For $\alpha+1$, assume that
  $U \subseteq Y$ is open such that
  $U \cap Y^{(\alpha)}_{f,\varepsilon}$ is $\varepsilon$-finite, say
  $\varepsilon$-$n$-finite.
  Let $y \in U \cap Y^{(\alpha)}_{f,\varepsilon}$.
  Then it will suffice to find $\pi^{-1}(y) \subseteq V$ open such that
  $V \cap X^{(\alpha)}_{f,\varepsilon}$ is $\varepsilon$-$n$-finite as well.

  Indeed, find $U' \subseteq Y$ open such that
  $y \in U' \subseteq \bar U' \subseteq U$, and let $V = \pi^{-1}(U')$.
  We claim that
  $V \cap X^{(\alpha)}_{f,\varepsilon}$ is $\varepsilon$-$n$-finite.
  If not, then there are $x_0,\ldots,x_n \in V \cap X^{(\alpha)}_{f,\varepsilon}$
  such that $i < j \leq n \Longrightarrow d(x_i,x_j) > \varepsilon$.
  By the homogeneity assumption there exists a precise automorphism
  $f$ of $X$ such that $\pi\circ f\rest_{\{x_0,\ldots,x_n\}}$ is isometric with
  image in $U$.
  Since $f$ is a precise automorphism it leaves $X^{(\alpha)}_{f,\varepsilon}$
  invariant, so $f(x_0),\ldots,f(x_n) \in X^{(\alpha)}_{f,\varepsilon}$ as well.
  By the induction hypothesis
  $\pi\circ f(x_0),\ldots,\pi\circ f(x_n) \in Y^{(\alpha)}_{f,\varepsilon}$, contradicting the assumption
  that $U \cap Y^{(\alpha)}_{f,\varepsilon}$ is $\varepsilon$-$n$-finite.
\end{proof}

\begin{thm}
  \label{thm:HomogEqualRank}
  Let $\pi\colon X \to Y$ be precise,
  and assume $X$ is homogeneous over $Y$.
  Then for all $K \subseteq Y$, $* \in \{f,b\}$ and $\varepsilon > 0$:
  $\CB^Y_{*,\varepsilon}(K) = \CB^X_{*,\varepsilon}(\pi^{-1}(K))$.

  Moreover, if this common rank is ordinal then:
  $\CBd^Y_{*,\varepsilon}(K) = \CBd^X_{*,\varepsilon}(\pi^{-1}(K))$.
\end{thm}
\begin{proof}
  Only the moreover part is left to be proved.
  Indeed, assume first that
  $\CB^Y_{f,\varepsilon}(K) = \CB^X_{f,\varepsilon}(\pi^{-1}(K)) = \alpha < \infty$.
  Let $X^{(\alpha)} = X_{f,\varepsilon}^{(\alpha)}$ and $Y^{(\alpha)} = Y_{f,\varepsilon}^{(\alpha)}$.
  By the first part $\pi(X^{(\alpha)}) = Y^{(\alpha)}$.

  Let $n = \CBd^Y_{f,\varepsilon}(K)$.
  Then there is an open set $K \subseteq U \subseteq Y$ such that
  $U \cap Y^{(\alpha)}$ is $\varepsilon$-$n$-finite.
  By the same argument as in the proof of \fref{lem:HomogIncCB} we
  find $V \subseteq X$ open such that $\pi^{-1}(K) \subseteq V$ and
  $V \cap X^{(\alpha)}$ is $\varepsilon$-$n$-finite so
  $n \geq \CBd^X_{f,\varepsilon}(\pi^{-1}(K))$.
  On the other hand, let $m = \CBd^X_{f,\varepsilon}(\pi^{-1}(K))$, so
  there is an open set $\pi^{-1}(K) \subseteq V \subseteq X$ such that
  $V \cap X^{(\alpha)}$ is $\varepsilon$-$m$-finite.
  Let $U = Y \setminus \pi(X \setminus V)$.
  Then $U$ is open, $K \subseteq U$, and $\pi^{-1}(U) \subseteq V$.
  It follows that
  $\pi^{-1}(U\cap Y^{(\alpha)}) \subseteq V \cap X^{(\alpha)}$,
  and since $\pi$ is precise
  $U \cap Y^{(\alpha)}$ is $\varepsilon$-$m$-finite,
  so $m \geq \CBd^Y_{f,\varepsilon}(K)$.

  The case $* = b$ is similar.
\end{proof}

Together with \fref{prp:HomogTypeSpace} this means we can define the
\emph{$\varepsilon$-Morley rank} of a type-definable set $X$ as
$\RM_\varepsilon(X) = \CB^{\tS_n(M)}_{f,\varepsilon}([X])$ where $M$ is any
approximately $\aleph_0$-saturated model containing the parameters for $X$.
Indeed, if $\bar M$ is the monster model then $\tS_n(\bar M)$ is
homogeneous over $\tS_n(M)$ whereby
$\CB^{\tS_n(M)}_{f,\varepsilon}([X]^{\tS_n(M)}) =
\CB^{\tS_n(\bar M)}_{f,\varepsilon}([X]^{\tS_n(\bar M)})$.
Similarly we define the \emph{$\varepsilon$-Morley degree} of $X$
as $\dM_\varepsilon(X) = \CBd^{\tS_n(M)}_{f,\varepsilon}([X])$.

The same is true for $\CB_{b,\varepsilon}$ ranks and degrees (which
coincide with the Morley ranks and degrees defined in
\cite{BenYaacov:Morley}).
However, $\CB_{f,\varepsilon}$ ranks seem to have the advantage of a more
natural notion of degree.

\subsection{Measures}

We conclude this section with a few results concerning Borel
probability measures on CB-analysable spaces.
Some of these results come from
joint work with Anand Pillay, whom we wish to thank for allowing
their inclusion here.

Recall that a measure on a topological space $X$ is \emph{regular}
if for every measurable $S$:
\begin{gather*}
  \mu(S)
  = \sup \{\mu(K)\colon S \supseteq K \text{ compact}\}
  = \inf \{\mu(U)\colon S \subseteq U \text{ open}\}.
\end{gather*}
Regular Borel measures on $X$ are in bijection with positive
integration functionals on $C(X,\bR)$ (or $C(X,[0,1])$).
Following our convention concerning terminology, the
\emph{Borel $\sigma$-algebra}
of a topometric space is the $\sigma$-algebra generated by
the topology.

\begin{thm}
  \label{thm:MeasureCptSupp}
  Let $X$ be a locally compact, CB-analysable topometric space,
  $\mu$ a regular Borel probability measure on $X$.
  Then for every $\varepsilon > 0$, $1-\varepsilon$
  of the mass of $\mu$ is supported by a
  metrically compact set.
\end{thm}
\begin{proof}
  Fix $r,\delta>0$, and let $U_\alpha = X\setminus X^{(\alpha)}_{f,r}$.
  Thus $U_{\alpha+1} \cap X^{(\alpha)}_{f,r}$
  is a union of open $r$-finite subsets
  of $X^{(\alpha)}_{f,r}$.

  Assume first that $\mu(U_1) > \delta$.
  By regularity there is a compact subset $F \subseteq U_1$ such that
  $\mu(F) > \mu(U_1) -  \delta$.
  Since $F \subseteq U_1$ is compact it is covered by finitely many $r$-finite
  sets, and is therefore $r$-finite.
  It follows that for arbitrary $\alpha$,
  if $\mu(U_{\alpha+1} \setminus U_\alpha) > \delta$ then there is a compact set
  $F \subseteq U_{\alpha+1} \setminus U_\alpha$
  such that
  $\mu(F) > \mu(U_{\alpha+1} \setminus U_\alpha) -\delta$.

  We now claim that
  $\mu(U_\alpha) =
  \sum_{\beta<\alpha} \mu(U_{\beta+1} \setminus U_\beta)$
  for all $\alpha$.
  In case $\alpha$ is countable this is just by $\sigma$-additivity, but for the
  general case a small inductive argument is required.
  $\alpha = 0$ (and indeed, $\alpha$ countable) is immediate,
  as is the successor case.
  For $\alpha$ limit, we have $U_\alpha = \bigcup_{\beta<\alpha} U_\beta$.
  Then every compact set $K \subseteq U_\alpha$
  is contained is some $U_\beta$, so by
  regularity:
  $\mu(U_\alpha) = \sup_{\beta < \alpha} \mu(U_\beta)$.

  By CB-analysability we have
  $\mu(X) = \sum_\alpha \mu(U_{\alpha+1} \setminus U_\alpha)$.
  Only countably many summands can be non-zero, and we may enumerate
  their indexes by $\{\alpha_n\colon n < \omega\}$ (not necessarily in order).
  For each $n$ find
  $F_n \subseteq U_{\alpha_n+1}\setminus U_{\alpha_n}$
  compact, $r$-finite and satisfying
  $\mu(F_n) > \mu(U_{\alpha_n+1}\setminus U_{\alpha_n}) - 2^{-n-1}\delta$.
  Then for some $m$: $\mu(\bigcup_{n<m}F_n) > \mu(X)-\delta$,
  and $K = \bigcup_{n<m}F_n$ is $r$-finite and compact, and therefore
  complete.

  Letting $r$ and $\delta$ vary, we see that for all $n$ we can find
  $K_n \subseteq X$ compact, $2^{-n}$-finite, satisfying
  $\mu(K_n) > \mu(X) - 2^{-n-1}\varepsilon$.
  Let $K = \bigcap K_n$.
  Then $K$ totally bounded and complete, i.e., it is metrically
  compact, and $\mu(K) > \mu(X)- \varepsilon$.
\end{proof}

\begin{cor}
  \label{cor:MeasureFiniteApprox}
  Let $X$ be a locally compact, CB-analysable topometric space, 
  $\mu$ a regular Borel probability measure on $X$.
  Then there is a sequence $\{\mu_n\colon n<\omega\}$ of probability
  measures with finite support converging weakly to $\mu$,
  i.e., $\int f\,d\mu_n \to \int f\,d\mu$ for every
  continuous function with compact support $f \in C_c(X,\bC)$.
  Moreover, if $\cF \subseteq C_c(X,\bC)$ is a family of uniformly bounded, and
  equally uniformly continuous functions then
  $\int f\,d\mu_n \to \int f\,d\mu$ uniformly for all $f \in \cF$.
\end{cor}
\begin{proof}
  For each $n < \omega$, choose $K_n \subseteq X$ metrically compact
  such that $\mu(K_n) > 1-2^{-n}$.
  By compactness we can find points
  $\{x_{n,i}\colon i < \ell_n\} \subseteq K_n$ such
  that $K_n \subseteq \bigcup {\overline B}(x_{n,i},2^{-n})$.
  For $i < \ell_n$ let
  $S_{n,i} = K_n\cap{\overline B}(x_{n,i},2^{-n-1})
  \setminus \bigcup_{j<i} S_{n,j}$.
  Then each $S_{n,j}$ is a Borel set, and they form a partition of
  $K_n$.
  Let $\mu_{n,0} = \sum_{i<\ell_n} \mu(S_{n,i})\delta_{x_{n,i}}$,
  where $\delta_x$ denotes
  the Dirac measure concentrated at $x$.
  Then $\mu_n = \mu(K_n)^{-1}\mu_{n,0}$ is a probability measure with finite
  support.

  Now let $f\colon X \to \bC$ be a continuous function with compact support.
  Then $f$ is uniformly continuous and bounded.
  Let $M = \sup |f|$, and for each $\delta > 0$ let
  $\Delta^{-1}(\delta) = \sup \{|f(x)-f(y)|\colon d(x,y) < \delta\}$.
  Then $\Delta^{-1}(\delta) \to 0$ as $\delta \to 0$, so:
  \begin{align*}
    \left| \int f\,d\mu - \int f \,d\mu_n \right|
    & \leq \sum_{i<\ell_n}
    \left| \int_{S_n,i} f\,d\mu - \int_{S_n,i} f \,d\mu_{n,0} \right| \\
    & \qquad + \left| \int_{X\setminus K_n} f\,d\mu \right|
    + \left| \int f\,d\mu_n - \int f \,d\mu_{n,0} \right| \\
    & \leq \sum_{i<\ell_n} \Delta^{-1}(2^{-n})\mu(S_{n,i}) + M2^{-n} + M2^{-n}\left(1-\frac{1}{\mu(K_n)}\right) \\
    & \leq \Delta^{-1}(2^{-n}) + 3\cdot2^{-n}M.
  \end{align*}
  This goes to $0$ as $n \to \infty$.
  The moreover part is implicit in the proof above.
\end{proof}

Measures on type spaces were originally studied by Keisler
\cite{Keisler:MeasuresAndForking}.
\begin{dfn}
  Let $T$ be a theory, $A$ a set of parameters.
  An $n$-ary \emph{Keisler measure} $\mu(\bar x)$
  over a set $A$ is a Borel
  probability measure on $\tS_n(A)$.
\end{dfn}
Such measures generalise the notion of a type (the Dirac measures
being in bijection with types).

Let $\mu(\bar x)$ be a Keisler measure over $A$.
Then a definable predicate $\varphi(\bar x)$ with parameters in $A$
is a continuous function $\varphi\colon \tS_n(A) \to[0,1]$, and we can calculate
its integral $I_\mu\varphi = \int \varphi^p\, d\mu(p)$.
Conversely, the integration functional $I_\mu$ determines $\mu$.

Let now $\mu(\bar x)$ be a Keisler measure over a model $M$.
Let $\varphi(\bar x,\bar y)$ be a definable predicate, possibly with some
parameters in $M$, and let us restrict our attention to instances
$\varphi(\bar x,\bar b)$ over $M$.
This defines a predicate $\bar b \mapsto I_\mu\varphi(\bar x,\bar b)$ on $M$,
which we denote by $I_{\mu(\bar x)}\varphi(\bar x,\bar y)$.

Let $\mu_\varphi$ be the image measure of $\mu$ on $\tS_\varphi(M)$.
Then $\mu_\varphi$ is a Borel probability measure as well, so we say it is a
\emph{Keisler $\varphi$-measure} over $M$.
Conversely, $\mu_\varphi$ determines $I_{\mu(\bar x)}\varphi(\bar x,\bar y)$,
so the family of all such $\mu_\varphi$ determines $\mu$
(in fact we only need the family of $\mu_\varphi$ where
$\varphi(\bar x,\bar y)$ are
formulae without hidden parameters).

\begin{cor}
  \label{cor:StabDefKeisMeasOverModel}
  Assume $\varphi(\bar x,\bar y)$ is a stable formula, $M$ a model, and
  $\mu_\varphi$ a $\varphi$-Keisler measure over $M$.
  Then $\mu_\varphi$ is definable, meaning that the
  predicate $I_{\mu(\bar x)} \varphi(\bar x,\bar y)$ is equal on $M$ to some
  definable predicate $\psi(\bar y)$ with parameters in $M$
  (it follows that $\psi$ is unique).

  Moreover, $\psi$ is a $\tilde \varphi$-predicate, namely a continuous
  function on $\tS_{\tilde \varphi}(M)$ where
  $\tilde \varphi(\bar y,\bar x) = \varphi(\bar x,\bar y)$.
\end{cor}
\begin{proof}
  By \cite{BenYaacov-Usvyatsov:CFO},
  $(\tS_\varphi(M),d_\varphi)$ is CB-analysable.
  By \fref{cor:MeasureFiniteApprox} $\mu$ is the weak limit of a
  sequence of finitely supported measures $\mu_n$.
  The family of all functions of the form
  $\varphi(\bar x,\bar b)$ is uniformly bounded (by $1$) and equally
  uniformly continuous (all are $1$-Lipschitz) on $\tS_\varphi(M)$.
  By the moreover part of \fref{cor:MeasureFiniteApprox}
  the predicates $I_{\mu_n(\bar x)}\varphi(\bar x,\bar y)$ converge uniformly
  to $I_{\mu(\bar x)}\varphi(\bar x,\bar y)$ on $M$.

  Let us write $\mu_n = \sum_{i<\ell_n} a_{n,i}\delta_{p_{n,i}}$.
  Each $p_{n,i}$ is definable, and let $\psi_{n,i}(\bar y)$ denote its
  definition.
  Set
  $$\psi_n(\bar y) = \sum_{i<\ell_n} a_{n,i}\psi_{n,i}(\bar y)
  = I_{\mu_n(\bar x)}\varphi(\bar x,\bar y).$$
  Then the definable predicates $\psi_n(\bar y)$ converge
  uniformly (on a model $M$ and therefore everywhere)
  and their limit is the desired definable predicate $\psi$.

  As each $\psi_{n,i}$ is a $\tilde \varphi$-predicate so is each
  $\psi_n$ and therefore $\psi$.
\end{proof}

Let us now look more closely at Keisler measures over a set
$A$ which is not a model.
Let $\varphi(\bar x,\bar y)$ be a definable predicate, possibly with some
parameters in $A$.
Recall from \cite[Section~6]{BenYaacov-Usvyatsov:CFO}
that a definable $\varphi$-predicate over $A$ is a definable predicate
$\chi(\bar x,C)$ which is at the same time over $A$ and
(equivalent to) a uniform limit of
continuous combinations of instances of $\varphi$ which need not be
over $A$.
We may write it as an infinitary continuous combination
$\chi(\bar x,C)
= \theta \circ (\varphi(\bar x,\bar c_i))_{i<\omega}$.
It may be more convenient to write
$\chi(\bar x,Y)
= \theta \circ (\varphi(\bar x,\bar y_i))_{i<\omega}$ and then replace
it with $\chi(\bar x,z)$ where $z$ is in the sort of canonical
parameters for
$\theta \circ (\varphi(\bar x,\bar y_i))_{i<\omega}$.
Call such a $\chi(\bar x,z)$ a \emph{$\varphi$-scheme}.
Thus a $\varphi$-predicate over $A$ can always be written as
$\chi(\bar x,c)$ where $\chi(\bar x,z)$ is a $\varphi$-scheme
and $c \in \dcl(A)$, and conversely every such
$\chi(\bar x,c)$ is a $\varphi$-predicate over $A$.

We define $\tS_\varphi(A)$ as the quotient of
$\tS_n(M)$ through which all the
$A$-definable $\varphi$-predicates factor.
It is thus a quotient both of $\tS_\varphi(M)$ and of
$\tS_n(A)$.
This construction does not depend on $M$ and yields a compact
Hausdorff space where the continuous mappings
$\tS_\varphi(A) \to [0,1]$ are precisely the $A$-definable
$\varphi$-predicates.
Applying \fref{lem:TopologicalQuotient} to the quotient mapping
$\tS_\varphi(M) \to \tS_\varphi(A)$ we may equip $\tS_\varphi(A)$ with
a topometric structure.
If $\varphi$ is stable then $\tS_\varphi(M)$ is CB-analysable and
by \fref{lem:CBImage} so is $\tS_\varphi(A)$.

A Keisler measure $\mu$ over $A$ gives rise to
predicates $\bar b \mapsto I_\mu\psi(\bar x,\bar b)$
on $\dcl(A)$,
which we denote by $I_{\mu(\bar x)}\varphi(\bar x,\bar y)$.
Let $\mu_\varphi$ be the image measure of $\mu$ on $\tS_\varphi(A)$.
Then $\mu_\varphi$ is a Borel probability measure as well, so we say it is a
\emph{$\varphi$-Keisler measure} over $A$.
The measure $\mu_\varphi$ determines,
and is determined by, the family of predicates
$I_{\mu(\bar x)} \chi(\bar x,z)$ defined above
where $\chi$ varies over all $\varphi$-schemes.
We say that $\mu_\varphi$ is definable if
all these predicates $I_{\mu(\bar x)} \chi(\bar x,z)$
are definable over $A$.
In case $A$ is a model this agrees with our earlier
notion of definable $\varphi$-measure in which we only
considered instances of $\varphi$.

Recall that the \emph{Galois group} $\Gal(A)$ of a set $A$ is
defined as $\Aut(\acl^{eq}(A)/A)$, namely the group of elementary
permutations of $\acl^{eq}(A)$ fixing $A$ point-wise.
This is a compact group in the topology of point-wise convergence
of its $A$-invariant action on $\acl^{eq}(A)$.

\begin{lem}
  \label{lem:UniqueMeasureACLExtension}
  Let $\mu$ be a Keisler measure over $A$.
  Let $G = \Gal(A)$ and let $H$ be its Haar measure.
  Then $\mu$ admits a unique extension to $\acl^{eq}(A)$
  which is invariant under the action of $G$.

  More precisely, if $\varphi(\bar x,\bar b)$ is a definable predicate
  over $\acl(A)$ then
  $\bar x \mapsto \int_G \varphi(\bar x,g\bar b)\, dH$
  is a $\varphi$-predicate over $A$ which we may denote by
  $\varphi^G(\bar x,\bar b)$.
  The unique invariant extension
  is then given by
  $I_{\tilde \mu}\varphi(\bar x,\bar b)
  = I_\mu \varphi^G(\bar x,\bar b)$.

  Similarly, a Keisler $\varphi$-measure $\mu_\varphi$ over $A$
  admits a unique $A$-invariant extension to $\acl^{eq}(A)$
  given by
  $I_{\tilde \mu}\chi(\bar x,c)
  = I_\mu \chi^G(\bar x,c)$
  for every $\varphi$-predicate $\chi(\bar x,c)$ over $\acl^{eq}(A)$.
\end{lem}
\begin{proof}
  If $\bar b \in \acl(A)$ then its orbit over $A$ is metrically
  compact.
  We can therefore approximate
  $\int_G \varphi(\bar x,g\bar b)\, dH$
  arbitrarily well by weighted sums
  $\sum_{i<k} a_i \varphi(\bar x,g_i\bar b)$, so it is indeed a
  $\varphi$-predicate.
  It is clearly $A$-invariant and therefore over $A$.
  In particular if $\chi(\bar x,c)$ is a $\varphi$-predicate
  over $\acl(A)$ then $\chi^G(\bar x,c)$ is a $\chi$-predicate,
  and therefore a $\varphi$-predicate, over $A$.
  Thus our definition of $\tilde \mu$ makes sense and it is not
  difficult to see that it does indeed define a
  regular Borel probability measure on $\tS_n(\acl^{eq}(A))$
  (or on $\tS_\varphi(\acl^{eq}(A))$).

  For uniqueness, assume that $\tilde \mu$ extends $\mu$ to
  $\acl^{eq}(A)$ and is $A$-invariant.
  Let $\varphi(\bar x,\bar b)$ be over $\acl^{eq}(A)$ and let
  $\sum_{i<k} a_i \varphi(\bar x,g_i\bar b)$
  be an $\varepsilon$-approximation of $\varphi^G(\bar x,\bar b)$.
  Then
  \begin{align*}
    I_{\tilde \mu} \varphi(\bar x,\bar b)
    &
    = \sum_{i<k} I_{\tilde \mu} \varphi(\bar x,g_i\bar b)
    = I_{\tilde \mu} \sum_{i<k} \varphi(\bar x,g_i\bar b)
    \\ &
    \approx_\varepsilon
    I_{\tilde \mu} \varphi^G(\bar x,\bar b)
    = I_\mu \varphi^G(\bar x,\bar b).
  \end{align*}
  Therefore
  $I_{\tilde \mu} \varphi(\bar x,\bar b)
  = I_\mu \varphi^G(\bar x,\bar b)$, as desired.
\end{proof}

\begin{cor}
  \label{cor:InvariantMeasureSupport}
  Let $A$ be an algebraically closed set,
  $M$ a strongly $|A|^+$-homogeneous and $|A|^+$-saturated model
  containing $A$.
  Let $\varphi$ be a stable formula and let
  $X \subseteq \tS_\varphi(M)$
  the collection of types which do not fork over $A$.
  Then every Keisler $\varphi$-measure $\mu_\varphi$ over $M$
  which is $A$-invariant is supported on $X$.
\end{cor}
\begin{proof}
  Assume not.
  Then by \fref{thm:MeasureCptSupp} there is a metrically compact set
  $Y \subseteq \tS_\varphi(M)$,
  $Y \cap X = \emptyset$, such that $\mu_\varphi(Y) > 0$.
  For each $p \in Y$ there exists $\varepsilon_p > 0$ and an infinite
  family $\{q_{p,i} \}$ of $A$-conjugates of $p$ such that
  $d(q_{p,i},q_{p,j}) \geq 2\varepsilon_p$ for all $i \neq j$.
  We may assume that $\varepsilon_p$ is maximal for which such a
  family exists.
  Notice that $\varepsilon_p \geq \varepsilon_{p'} - d(p,p')$
  for all $p,p' \in Y$.
  Since $Y$ is metrically compact it admits a countable dense subset
  $\{p_i\}_{i<\omega}$.
  For $p \in Y$ there exists some $p_i \in B(p,\varepsilon_p/2)$.
  Thus $\varepsilon_{p_i} > \varepsilon_p/2$ and
  $p \in B(p_i,\varepsilon_{p_i})$.
  In other words, $Y = \bigcup_{i<\omega} B(p_i,\varepsilon_{p_i})$.
  Since $\mu(Y) > 0$ there is $p$ such that
  $\mu_\varphi(B(p,\varepsilon_p)) = r > 0$.
  Then $\mu_\varphi(B(q_{p,i},\varepsilon_p)) = r$ for each
  $A$-conjugate $q_{p,i}$ of $p$ in the family we chose earlier.
  Since the sets $B(q_{p,i},\varepsilon_p)$ are disjoint
  $\mu$ has infinite total measure, a contradiction.
  This contradiction concludes the proof.
\end{proof}

We can now improve \fref{cor:StabDefKeisMeasOverModel}.

\begin{cor}
  \label{cor:StabDefKeisMeasOverSet}
  (Compare with Keisler \cite[Theorem~2.1]{Keisler:ChoosingElements}.)\\
  Assume $\varphi(\bar x,\bar y)$ is a stable formula, $A$ a set, and
  $\mu_\varphi$ a $\varphi$-Keisler measure over $A$.
  Then $\mu_\varphi$ is definable.
  Moreover, it admits a unique definition scheme over $A$
  which defines a Keisler $\varphi$-measure over any set containing
  $A$, and this definition scheme necessarily consists of
  $\tilde \varphi$-predicates.
\end{cor}
\begin{proof}
  By \fref{lem:UniqueMeasureACLExtension}
  we may assume that $A = \acl^{eq}(A)$ is algebraically closed.
  Fix a very saturated and homogeneous  model $M \supseteq A$
  and let $X \subseteq \tS_\varphi(M)$ be the collection of
  $\varphi$-types which do not fork over $A$, i.e., which are
  definable over $A$.
  The restriction mapping
  $\pi\colon X \to \tS_\varphi(A)$ is a homeomorphism, so we may use
  it to pull $\mu_\varphi$ to a regular Borel probability measure
  $\tilde \mu_\varphi$ on $\tS_\varphi(M)$ which is supported on
  $X$.
  By \fref{cor:StabDefKeisMeasOverModel} $\tilde \mu_\varphi$ is
  definable.
  Since automorphisms which fix $A$ necessarily fix every point in
  $X$ they fix $\tilde \mu_\varphi$ and therefore fix its definition.
  Therefore $\tilde \mu_\varphi$ is definable over $A$.

  For uniqueness, assume that $\tilde \mu_\varphi'$ is another
  $A$-invariant extension of $\mu_\varphi$.
  By \fref{cor:InvariantMeasureSupport} it must be supported on $X$.
  Now fix $\bar b \in M$ and let
  $q(\bar y) \in \tS_{\tilde \varphi}(M)$ be the unique non-forking
  extension of $\tp_{\tilde \varphi}(\bar b/A)$ to $M$,
  $\psi(\bar x)$ is definition.
  Then $\psi$ is a $\varphi$-predicate over $A$.
  By forking symmetry, if $p \in X$ then
  $\varphi(\bar x,\bar b)^p = \psi(\bar x)^p$.
  Therefore
  \begin{align*}
    I_{\tilde \mu'} \varphi(x,\bar b) &
    = \int \varphi(x,\bar b)^p d\tilde \mu_\varphi'
    = \int_X \varphi(x,\bar b)^p d\tilde \mu_\varphi'
    \\ &
    = \int_X \psi(x)^p d\tilde \mu_\varphi'
    = \int \psi(x)^p d\tilde \mu_\varphi'
    \\ &
    = I_{\tilde \mu'} \psi(x,\bar b)
    = I_\mu \psi(x,\bar b) = \ldots
    = I_{\tilde \mu} \varphi(x,\bar b).
    \qedhere
  \end{align*}
\end{proof}

Given the uniqueness of the good definitions we may always regard
$I_\mu\varphi(\bar x,\bar y)$ as $\tilde \varphi$-predicate over the
parameter set of $\mu$  (or of $\mu_\varphi$) without ambiguity.

\begin{cor}
  \label{cor:Fubini}
  (Compare with Keisler \cite[Corollary~6.16]{Keisler:MeasuresAndForking}.)\\
  Let $A$ be a set and
  let $\varphi(\bar x,\bar y)$
  be any stable formula (or even definable predicate,
  possibly with hidden parameters in $A$).
  Let $\mu(\bar x)$ and $\nu(\bar y)$ be two Keisler measures over $A$.
  Then Fubini's Theorem holds for $\mu$ and $\nu$:
  $$I_{\mu(\bar x)}I_{\nu(\bar y)}\varphi(\bar x,\bar y) =
  I_{\nu(\bar y)}I_{\mu(\bar x)}\varphi(\bar x,\bar y).$$
\end{cor}
\begin{proof}
  We know by \fref{cor:StabDefKeisMeasOverSet} that
  $I_{\nu(\bar y)}\varphi(\bar x,\bar y)$ is a
  $\varphi$-predicate and $I_{\mu(\bar x)}\varphi(\bar x,\bar y)$ is a
  $\tilde \varphi$-predicate, so in fact the statement is only concerned
  with $\mu_\varphi$ and $\nu_{\tilde \varphi}$.
  In case these are Dirac measures, i.e., complete types, this is just
  standard forking symmetry:
  \begin{align*}
    I_{p(\bar x)}I_{q(\bar y)}\varphi(\bar x,\bar y)
    & = I_{p(\bar x)} d_q\tilde \varphi(\bar x) = d_q\tilde \varphi(\bar x)^p \\
    & = d_p\varphi(\bar y)^q = I_{q(\bar y)} d_p\varphi(\bar y)
    = I_{q(\bar y)}I_{p(\bar x)}\varphi(\bar x,\bar y).
  \end{align*}
  In the general case proceed use \fref{cor:MeasureFiniteApprox}
  to approximate $\mu_\varphi$ and $\nu_{\tilde \varphi}$ by finite sums of complete types with
  weights, and apply the first case.
\end{proof}

Note that if $T$ is stable we can then define a Keisler measure
$(\mu\times\nu)(\bar x,\bar y)$ over $A$ by:
$I_{\mu\times\nu}\varphi(\bar x,\bar y) =
I_{\mu(\bar x)}I_{\nu(\bar y)}\varphi(\bar x,\bar y)$.
This is the free product of $\mu$ and $\nu$, generalising free product
of types.

\section{Perturbation metrics}
\label{sec:PertMet}

In this section we apply earlier results to questions around
perturbations of continuous structures, originally studied in
\cite{BenYaacov:Perturbations}.
We shall therefore leave the abstract setting and deal
exclusively with topometric structures on type spaces
$\tS_n(T)$ (or $\tS_n(A)$) in the context of
continuous logic.

\subsection{Definitions and characterisations}

We recall from \cite{BenYaacov:Perturbations}
that a \emph{perturbation system}
$\fp$ for a theory $T$ can be given by a system of
\emph{perturbation metrics} $d_{\fp,n}$ on $\tS_n(T)$ for each $n < \omega$ such
that $n \mapsto (\tS_n(T),d_\fp)$ is a precise topometric functor, namely:
\begin{enumerate}
\item Each $(\tS_n(T),d_\fp)$ is a topometric space.
\item For every $n,m < \omega$ and mapping $\sigma\colon n \to m$, the corresponding
  mapping $\sigma^*\colon \tS_m(T) \to \tS_n(T)$ is a precise morphism of
  topometric spaces.
\end{enumerate}

We shall follow the notation from \cite{BenYaacov:Perturbations} and denote
${\overline B}_{d_\fp}(X,r)$ by $X^{\fp(r)}$.

\begin{dfn}
  Let $\fp$ be a perturbation system for $T$, $M,N \models T$, and $r \geq 0$.
  \begin{enumerate}
  \item A \emph{partial $\fp(r)$-perturbation} from $M$ to $N$
    is a partial mapping
    $f\colon M \dashrightarrow N$ such that for all
    $\bar a \in \dom(f)$: $d_\fp(\tp_M(\bar a),\tp_N(f(\bar a))) \leq r$.
  \item If $f$ above is bijective then it is a
    \emph{$\fp(r)$-perturbation} of $M$ into $N$.
    We denote the set of all such mappings by
    $\Pert_{\fp(r)}(M,N)$.
  \end{enumerate}
\end{dfn}

\begin{rmk}
  What we call a perturbation here was called a \emph{bi-perturbation}
  in \cite{BenYaacov:Perturbations}, the term perturbation being reserved
  there for the
  somewhat weaker notion of a total (but not necessarily surjective)
  partial perturbation.
  The distinction is more important when dealing with asymmetric
  perturbation radii and pre-radii with which much of that paper was
  concerned and which do not appear here at all.
  We apologise for the inconvenience.
\end{rmk}

The preciseness of $\sigma^*$ has a concrete meaning for various special
cases for $\sigma$:
\begin{itemize}
\item 
  Preciseness of $\sigma^*$ when $\sigma\colon 2\to1$ is the unique
  mapping is equivalent to the property that if $p(x,y) \in \tS_2(T)$ then
  either $p \in [x=y]$ or $d_\fp(p,[x=y]) = \infty$ (since $p \notin [x=y]$ implies that
  $d(q,(\sigma^*)^{-1}(p)) = d(q,\emptyset) = \infty$ for all $q \in \tS_1(T)$).
  By compactness it follows that for every $r>0$ and $\varepsilon > 0$ there is
  $\delta > 0$ such that
  $[d(x,y) <\delta]^{\fp(r)} \subseteq [d(x,y) \leq \varepsilon]$,
  and by symmetry:
  $[d(x,y) >\varepsilon]^{\fp(r)} \subseteq [d(x,y) \geq \delta]$.
  In particular every partial
  $\fp(r)$-perturbation is uniformly continuous and
  injective, and its inverse is a partial
  $\fp(r)$-perturbation by
  symmetry of $d_\fp$.
\item Preciseness of $\sigma^*$ when $\sigma\colon n \to n$ is a permutation just
  means that a permutation of the variables is an isometry of
  $(\tS_n(T),d_\fp)$, i.e., the notion of perturbation does not depend
  on the order of an enumeration.
\item Preciseness of $\sigma^*$ when $\sigma\colon n \to n +1$ is the inclusion tells
  us that a partial $\fp(r)$-perturbation can be extended to one more
  element.
\end{itemize}

Along with a standard back-and-forth argument this yields:
\begin{fct}
  \label{fct:PertExt}
  Let $\fp$ be a perturbation system for $T$, $r > 0$.
  Let $M,N \models T$, $\bar a \in M^n$, $\bar b \in N^n$.
  Then the following are equivalent:
  \begin{enumerate}
  \item $d_\fp(\tp_M(\bar a),\tp_N(\bar b)) \leq r$.
  \item There are $M' \succeq M$, $N' \succeq N$ and
    $f \in \Pert_{\fp(r)}(M',N')$ such that
    $f(\bar a) = \bar b$.
  \end{enumerate}
\end{fct}

This means that the perturbation system is determined by the mapping
$(M,N,r) \mapsto \Pert_{\fp(r)}(M,N)$, and it will be useful to give a
general characterisation of mappings of this form.
\begin{thm}
  \label{thm:PertChar}
  Let $\fp$ be a perturbation system for $T$.
  Then for each $r \in \bR^+$ and $M,N \in \Mod(T)$,
  $\Pert_{\fp(r)}(M,N)$ is a set of bijections of $M$ with $N$
  satisfying the following properties:
  \begin{enumerate}
  \item Monotonicity:
    $\Pert_{\fp(r)}(M,N) = \bigcap_{s>r} \Pert_{\fp(s)}(M,N)$.
  \item Strict reflexivity:
    $\Pert_{\fp(0)}(M,N)$ is the set of isomorphisms
    of $M$ with $N$.
  \item Symmetry:
    $f \in \Pert_{\fp(r)}(M,N)$ if and only $f^{-1} \in \Pert_{\fp(r)}(N,M)$.
  \item Transitivity:
    if
    $f \in \Pert_{\fp(r)}(M,N)$ and
    $g \in \Pert_{\fp(s)}(N,L)$ then
    $g \circ f \in \Pert_{\fp(r+s)}(M,L)$.
  \item Uniform continuity:
    for each $r \in \bR^+$, all members of $\Pert_{\fp(r)}(M,N)$, where
    $M,N$ vary over all models of $T$, satisfy a common modulus of
    uniform continuity.
  \item Ultraproducts:
    If $f_i \in \Pert_{\fp(r)}(M_i,N_i)$ for $i \in I$, and $\sU$ is
    an ultrafilter on $I$ then
    $\prod_\sU f_i \in \Pert_{\fp(r)}\bigl( \prod_\sU M_i, \prod_\sU N_i \bigr)$.
    (Note that $\prod_\sU f_i$ exists by the uniform continuity
    assumption).
  \item Elementary substructures:
    If $f \in \Pert_{\fp(r)}(M,N)$, $M_0 \preceq M$, and
    $N_0 = f(M_0) \preceq N$ then
    $f\rest_{M_0} \in \Pert_{\fp(r)}(M_0,N_0)$.
  \end{enumerate}
  Conversely, every mapping associating to
  every triplet $(r,M,N) \in \bR^+ \times \Mod(T)^2$ a set of bijections
  $\Pert'_r(M,N)$ satisfying the properties above is of the form
  $(r,M,N) \mapsto \Pert_{\fp(r)}(M,N)$ for a unique perturbation system
  $\fp$.
\end{thm}
\begin{proof}
  The first part is fairly immediate from facts we already know, so we
  only prove the converse.
  The uniqueness part follows from \fref{fct:PertExt} so we prove
  existence.

  Say that
  ``$d_\fp(p,q) \leq r$'' (in quotes, since we have not yet given a
  value to $d_\fp(p,q)$) if there are models $M,N \models T$
  and $f \in \Pert'_r(M,N)$ sending a realisation of $p$ to one of $q$.

  First we claim that ``$d_\fp(p,q) \leq r$'' if and only if
  ``$d_\fp(p,q) \leq s$'' for all $s > r$.
  Left to right is immediate from monotonicity.
  For right to left, let $s_n \searrow r$.
  For each $n$ let $f_n \in \Pert'_{s_n}(M_n,N_n)$ witness that
  ``$d_\fp(p,q) \leq s_n$'' sending $\bar a_n \in M_n$ to $\bar b_n \in N_n$.
  Let $\sU$ be a non-principal ultrafilter on $\omega$, and let
  $(M,N,f) = \prod_\sU (M_n,N_n,f_n)$.
  Then $f \in \Pert'_{s_n}(M,N)$ for all $n$ by the ultraproduct
  property and $f([a_n]) = [b_n]$,
  whereby $f \in \Pert'_r(M,N)$ witnesses that ``$d(p,q) \leq r$''.

  We may therefore define
  $d_\fp(p,q) = \inf \{r\colon \text{``$d_\fp(p,q) \leq r$''}\} \in [0,\infty]$ and
  drop the quotes.
  Strict reflexivity tells us that
  $d_\fp(p,q) = 0 \Longleftrightarrow p = q$, and
  symmetry and transitivity imply symmetry of $d_\fp$ and the triangle
  inequality, so $d_\fp$ is a metric.
  The ultraproduct property implies that
  $\{(p,q) \in \tS_n(T)^2\colon d_\fp(p,q) \leq r\}$ is closed,
  and $(\tS_n(T),d_\fp)$ is a topometric space for all $n$.

  For preciseness, let now $\sigma\colon n \to m$, $p \in \tS_m(T)$, $q \in \tS_n(T)$,
  and we need to show that
  $d_\fp(p,(\sigma^*)^{-1}(q)) = d_\fp(\sigma^*(p),q)$.
  Assume first that $d_\fp(p,(\sigma^*)^{-1}(q)) = r < \infty$,
  so $d_\fp(p,q') = r$ for some $q' \in (\sigma^*)^{-1}(q)$.
  Let $f \in \Pert'_r(M,N)$ witness this, sending
  $\bar a \models p$ (of length $m$) to $f(\bar a) \models q'$.
  Let $a'_i = a_{\sigma(i)}$ for $i<n$.
  Then $\bar a' \models \sigma^*(p)$ and $f(\bar a') \models q$, whereby
  $d_\fp(\sigma^*(p),q) \leq r$.
  Conversely, assume $d_\fp(\sigma^*(p),q) = r < \infty$.
  Let $f \in \Pert'_r(M,N)$ witness this, sending
  $\bar a \models \sigma^*(p)$ (of length $m$) to $f(\bar a) \models q$.
  By the ultraproduct property we may replace $M$ with an
  $\aleph_1$-saturated elementary extension, in which there is a
  tuple $c_{<m}$ such that $a_i = c_{\sigma(i)}$ for $i < n$.
  Then $f(\bar c)$ realises a type in $(\sigma^*)^{-1}(q)$, showing that
  $d_\fp(p,(\sigma^*)^{-1}(q)) \leq r$.
  Equality follows.

  We have shown that $d_\fp$ is indeed the perturbation metric
  associated to a perturbation system $\fp$.
  The inclusion $\Pert'_r(M,N) \subseteq \Pert_{\fp(r)}(M,N)$ is immediate
  from the construction.
  Assume now that $f \in \Pert_{\fp(r)}(M,N)$.
  Let $A \subseteq M$ be a finite subset and enumerate it as a tuple $\bar a$.
  Let $p = \tp(\bar a)$, $q = \tp(f(\bar a))$, so
  $d_\fp(p,q) \leq r$.
  This is witnessed by some $f_A \in \Pert'_r(M_A,N_A)$
  sending $\bar a' \models p$ to $f_A(\bar a') \models q$.
  In other words, there are partial elementary mappings
  $\theta_A\colon M \dashrightarrow M_A$ and
  $\theta'_A\colon N \dashrightarrow N_A$, with
  domains $A$ and $f(A)$, respectively,
  such that $f_A \circ \theta_A = \theta'_A \circ f$ on $A$.
  Set $I = \{A\subseteq M\colon |A| < \infty\}$
  and let $\sU$ be an ultrafilter on $I$
  containing the set $\{A \in I\colon a \in A\}$ for each $a \in M$.
  Let $(\tilde M,\tilde N, \tilde f) = \prod_\sU (M_A,N_A,f_A)$, so
  $\tilde f \in \Pert'_r(\tilde M,\tilde N)$ by the ultraproduct
  property.
  Define $\theta\colon M \to \tilde M$ by $\theta(a) = [a_A]_{A \in I}$ where
  $a_A = \theta_A(a)$ if $a \in A$, and anything otherwise.
  Define $\theta'\colon N \to \tilde N$ similarly.
  Then $\theta$ and $\theta'$ are (total) elementary embeddings,
  and $\tilde f \circ \theta = \theta' \circ f$.
  It follows that $f \in \Pert'_r(M,N)$ by the elementary
  substructures property, as desired.
\end{proof}

\begin{rmk}
  Let $\cL^2_f$ be the language for triplets $(M,N,f)$,
  consisting of two disjoint copies of $\cL$ plus a new
  function symbol $f$ going from one to the other.
  Then the ultraproduct and elementary
  substructures properties are \emph{not} equivalent
  (modulo previous axioms) to
  the elementarity of the class
  \begin{gather*}
    \{(M,N,f)\colon M,N \models T, f \in \Pert_\fp(M,N)\}.
  \end{gather*}
  Indeed, The assumption that $(M,N,f) \subseteq (M',N',f')$ and
  $M \preceq M'$, $N \preceq N'$ does not imply that
  $(M,N,f) \preceq (M',N',f')$, so elementarity is not strong enough to
  imply the elementary substructures property.

  The correct equivalent assumption is that the class above is
  elementary and that its theory is ``universal over $\cL$''.
  In case $T$ eliminates quantifiers (which we may always assume) this
  just means that it is given by:
  \begin{gather*}
    \{M \models T\} \cup \{N \models T\} \cup \{\text{some universal axioms (involving both
      copies and $f$)}\}.
  \end{gather*}
\end{rmk}

\subsection{Extensions of perturbation systems}

By definition, a perturbation system $\fp$ for a theory $T$
compares complete types \emph{without} parameters, telling us
the by how much one needs to be perturbed in order to obtain the
other.
What about types \emph{with} parameters?
Adding a parameter in a set $A \subseteq M$ to the language consists of two
steps: adding new constant symbols to the language for the members of
$A$, and replacing $T$ with $T(A) = \Th_{\cL(A)}(M)$.
Given a perturbation system for $T$ in $\cL(A)$, the second step
merely consists of the restriction to a smaller family of type, so let
us concentrate on the first step.
More generally, let us explore the extension of perturbation systems
to a bigger language $\cL' \supseteq \cL$.
Replacing a function symbol $f(\bar x)$ with the predicate
$G_f(\bar x,y) = d(f(\bar x),y)$, we may assume only predicate symbols
are added.

Let us first consider the case where a single new symbol is added:
$\cL_P = \cL \cup \{P\}$.
Let $T'$ denote the set of $\cL_P$-consequence of $T$ (i.e., $T$ viewed
as an $\cL_P$-theory).
Then it is natural (to the author, at least, which is what counts at
the moment) to extend $\fp$ to a perturbation system $\fp_P$ for
$T'$ allowing small perturbations of
$P$.
Thus for every $(M,P^M),(N,P^N) \models T'$ we define:
\begin{gather*}
  \Pert_{\fp_P(r)}\bigl( (M,P^M),(N,P^N) \bigr)
  =
  \left\{
    \theta \in \Pert_{\fp(r)}(M,N)\colon
    \begin{aligned}[c]
      & \text{for all $\bar b\in M$:}  \\
      & |P^M(\bar b)-P^N(\theta(\bar b))| \leq r
  \end{aligned}
  \right\}.
\end{gather*}
It is fairly straightforward to verify that this definition satisfies
the list of properties from \fref{thm:PertChar}, and thus indeed
defines a perturbation system $\fp_P$.

In case we wish to add several new symbols $\bar P = \{P_i\colon i < k\}$, we
merely iterate this construction.
\begin{multline*}
  \Pert_{\fp_{\bar P}(r)}\bigl( (M,P_0^M,\ldots),(N,P_0^N,\ldots) \bigr) = \\
  \left\{
    \theta \in \Pert_{\fp(r)}(M,N)\colon
    \begin{aligned}[c]
      & \text{for all $i < k$ and $\bar b\in M$:}\\
      & |P_i^M(\bar b)-P_i^N(\theta(\bar b))| \leq r
    \end{aligned}
  \right\}.
\end{multline*}
This is particularly elegant as it does not depend
on the order in which we add the symbols.
However, one could come up with several variants of this definition,
such as:
\begin{multline*}
  \Pert_{\fp'_{\bar P}(r)}\bigl( (M,P_0^M,\ldots),(N,P_0^N,\ldots) \bigr) = \\
  \left\{
    \theta \in \Pert_{\fp(r)}(M,N)\colon
    \begin{aligned}[c]
      & \text{for all $i < k$ and $\bar b\in M$:}\\
      & |P_i^M(\bar b)-P_i^N(\theta(\bar b))| \leq 2^{-i}r
    \end{aligned}
  \right\}.
\end{multline*}
Or:
\begin{multline*}
  \Pert_{\fp''_{\bar P}(r)}\bigl( (M,P_0^M,\ldots),(N,P_0^N,\ldots) \bigr) = \\
  \left\{
    \theta \in \Pert_{\fp(r)}(M,N)\colon
    \begin{aligned}[c]
      & \text{for all $i < k$ and $\bar b\in M$:}\\
      & |P_i^M(\bar b)-P_i^N(\theta(\bar b))| \leq 2^ir
    \end{aligned}
  \right\}.
\end{multline*}
As long as we only add finitely many symbols, all three definitions
are equivalent, in the sense that the metrics $d_{\fp_{\bar P}}$,
$d_{\fp'_{\bar P}}$ and $d_{\fp''_{\bar P}}$ are all uniformly
equivalent metrics.
Of course, one can come up with many more variants of this kind, but as
long as we allow to perturb each of the finitely many new symbols, we
are always going to get something equivalent to $\fp_{\bar P}$, which,
as we said, seems the most elegant of the lot.

Let us now consider the case where countably many new symbols are
added.
All three constructions suggested above admit an obvious
generalisation to $\bar P = \{P_i\colon i < \omega\}$.
However, an essential distinction now presents itself between
$\fp_{\bar P}$, $\fp'_{\bar P}$ on the one hand, and
$\fp''_{\bar P}$ on the other.

Indeed, $\fp''_{\bar P}$ is a very relaxed perturbation system, as a
positive perturbation distance only takes into account
finitely many of $\{P_i\colon i < \omega\}$.
More precisely, the question whether or not
$\theta \in \Pert_{\fp''_{\bar P}(r)}(M,N)$ depends only on
$\{P_i\colon 0 \leq i < -\log_2 r\}$.
This is essentially the only way of getting a
non trivial perturbation
system for first order logic.
For example, let $T$ be a theory in a countable language $\cL$.
Let $\fp$ be the trivial perturbation system for the
language of equality, and let $\fp''_\cL$ be as above.
Then $T$ is $\fp''_\cL$-$\aleph_0$-categorical if and only if every
restriction of $T$ to a finite sub-language is $\aleph_0$-categorical.
In the case of an uncountable tuple $\bar P$, in order to obtain a
``perturbation system'' with similar properties we should have
to replace metrics with non-metrisable uniform structures.
At the moment we do not see the point in doing so,
as the usefulness of $\fp''_{\bar P}$ for infinite $\bar P$
is not at all clear.
In particular, type spaces over infinitely many parameters in classical
logic do not involve non-trivial perturbation systems, so
the ``correct'' way to extend $\fp$ to types over an infinite set $A$
should go through another construction.

In contrast,
$\fp_{\bar P}$ and $\fp'_{\bar P}$ can be arbitrarily strict when
applied judiciously to infinitely many new symbols.
Indeed, in the case of $\fp'_{\bar P}$, we may enumerate $\bar P$ with
repetitions, repeating each symbols infinitely many times
(or, if this bothers the reader, we could enumerate many copies of
each new symbol and later add to $T$ the axioms that all copies of a
single symbol coincide).
In that case, a $\fp'_{\bar P}(r)$-perturbation of
$\cL_{\bar P}$-structures would necessarily fix the interpretation of
every new predicate symbol $P_i$.
In case of the apparently more relaxed $\fp_{\bar P}$, the same can be
achieved by replacing each new symbol $P_i$ with a tree of symbols
$\{P_i^\sigma\colon \sigma \in 2^{<\omega}\}$,
viewing $P_i^\emptyset$ as $P_i$, and adding
axioms that $P_i^{\sigma0} = 2P_i^\sigma \wedge 1$ and
$P_i^{\sigma1} = (2P_i^\sigma - 1) \vee 0$.
Then for $r < 1$ a $\fp_{\bar P}(r)$-perturbation of models of these
axioms would necessarily fix all the new symbols.

The somewhat philosophical discussion in the previous paragraph is
meant to convince the reader that when adding infinitely many
new symbols to the language, the most reasonable (and canonical) way
of extending $\fp$ is by fixing all the new symbols.
In that case we might as well define directly the extension
$\fp\rest_{\bar P}$ (``$\fp$ \emph{over} $\bar P$'') by:
\begin{multline*}
  \Pert_{\fp\rest_{\bar P}(r)}\bigl( (M,\bar P^M),(N,\bar P_0^N) \bigr) = 
  \left\{
    \theta \in \Pert_{\fp(r)}(M,N)\colon
    P^M = P^N \circ \theta \text{ for all $P \in \bar P$}
  \right\}.
\end{multline*}
In particular, when we extend $\fp$ to $\tS_n(A)$ where $A$ is
infinite we shall use $\fp\rest_A$.

Let us now re-examine the distance between two types
$p,q \in \tS_n(T)$.
If $\fp$ is a perturbation system then $d_\fp(p,q)$
measures by how much a realisation of $p$ needs to be perturbed in
order to get a realisation of $q$.
But we can also identify $p$ and $q$ with completions of $T$ in the
language $\cL(\bar c)$, where $\bar c$ is an $n$-tuple of new constant
symbols, which we denote by $p',q' \in \tS_0^{\cL(\bar c)}(T)$.
Let $\id$ be the trivial perturbation system for $T$, and let
$\id_{\bar c}$ be constructed as above.
Then $\Pert_{\id_{\bar c}}((M,\bar a),(N,\bar b))$ consists of all
isomorphisms of $\theta\colon M \to N$ such that $d(\bar b,\theta(\bar a)) \leq r$,
and $d_{\id_{\bar c}}(p',q')$ is simply the standard distance
$d(p,q)$, measuring by how much a realisation of $p$ needs to be moved
in order to obtain a realisation of $q$.
Finally, we can combine both constructions defining
$\tilde d_\fp(p,q) = d_{\fp_{\bar c}}(p',q')$.
We obtain a notion of
distance which allows both to perturb the underlying structure and to
move the realisations.
We could also define it directly (as was done in \cite{BenYaacov:Perturbations}) as:
\begin{gather*}
  \tilde d_\fp(p,q) =
  \inf \left\{
    r\geq 0\colon
    \begin{aligned}[c]
      & \Bigl(\exists\, M \models p(\bar a),\, N \models q(\bar b),\,
      \theta \in \Pert_{\fp(r)}(M,N) \Bigr) \\
      & \qquad \Bigl(\forall\, c \in M,\, i<n\Bigr)
      \Bigl( |d^M(c,a_i) - d^N(\theta(c),b_i)| \leq r \Bigr)
    \end{aligned}
  \right\}.
\end{gather*}

In the terminology of \fref{sec:Isol} we can restate
\cite[Theorem~3.8 and Proposition~3.9]{BenYaacov:Perturbations} as:
\begin{quote}
  A complete countable theory $T$ is $\fp$-$\aleph_0$-categorical if and
  only if every finite tuple $\bar a$, every type
  $p \in \tS_1(\bar a)$ is weakly
  $\tilde d_{\fp_{\bar a}}$-isolated.
\end{quote}
And:
\begin{quote}
  If $T$ is a complete countable theory and every type
  $p \in \tS_n(T)$ is $\tilde d_\fp$-isolated, then $T$ is
  $\fp$-$\aleph_0$-categorical.
\end{quote}

\begin{rmk}
  C.\ Ward Henson pointed out that if $(X,d)$ is a complete
  topometric space, then the following are equivalent:
  \begin{enumerate}
  \item Every point $x \in X$ is weakly $d$-isolated.
  \item The set of $d$-isolated points in $d$-dense in $X$.
  \end{enumerate}
  This is a special case of \fref{lem:WIsolRank}.
\end{rmk}

\subsection{$\lambda$-stability up to perturbation}

In the course of studying metric structures one encounters many which
should, according to all moral standards, be $\aleph_0$-stable (or at least
superstable), but are not.
Examples for this are probability spaces with a generic automorphism
\cite{BenYaacov-Berenstein:HilbertProbabilityAutmorphismPerturbation}
or Nakano spaces \cite{BenYaacov:NakanoSpaces}.
Reassuringly enough, both turn out to be $\aleph_0$-stable up to a natural
perturbation system.
Our earlier work allows us to conclude almost immediately that this
notion of $\aleph_0$-stability, and more generally, of $\lambda$-stability,
satisfies some expected properties.
In particular, $\aleph_0$-stability coincides with the existence
of appropriate Morley ranks.
As pointed out in the introduction, definitions and results of
Iovino \cite{Iovino:StableBanach} can be viewed as precursors to some
presented here.

\begin{conv}
  Henceforth, when $\fp$ is a perturbation system for $T$ and $A$ a
  set of parameters, we always interpret $d_\fp$ on $\tS_n(A)$ as
  $d_{\fp\rest_A}$, and accordingly, $\tilde d_\fp$ as
  $\tilde d_{\fp\rest_A}$.
\end{conv}

\begin{dfn}
  Let $T$ be a theory, $\lambda \geq |\cL|$, and
  $\fp$ a perturbation system for $T$.
  We say that $T$ is $\fp$-$\lambda$-stable if
  $\|(\tS_n(A),\tilde d_\fp)\| \leq \lambda$
  whenever $|A| \leq \lambda$.
\end{dfn}

First of all it should be pointed out that for any
perturbation system $\fp$, $\fp$-$\lambda$-stability is weaker than
$\lambda$-stability, since $\tilde d_\fp$ is always coarser than the
standard metric $d$ (which coincides with $\tilde d_{\id}$).
We thus need to make sure that $\fp$-$\lambda$-stability is still strong
enough to imply stability.

\begin{fct}
  Let $T$ be a theory, $\lambda \geq |\cL|$, and
  $\fp$ a perturbation system for $T$.
  Then $T$ is $\fp$-$\lambda$-stable if and only if for any model $M \models T$:
  $\|M\| = \lambda \Longrightarrow \|(\tS_n(M),\tilde d_\fp)\| = \lambda$.
\end{fct}
\begin{proof}
  By L\"owenheim-Skolem and the fact that if $A \subseteq M$ is dense then
  $(\tS_n(A),\tilde d_\fp) \cong (\tS_n(M),\tilde d_\fp)$.
\end{proof}

\begin{lem}
  Let $T$ be a theory, $\fp$ a perturbation system for $T$,
  and $M \models T$.
  Let $\varphi(\bar x,\bar y)$ be any formula, $|\bar x| = n$.
  Then
  $\pi_\varphi\colon (\tS_n(M),\tilde d_\fp)
  \to (\tS_\varphi(M),d_\varphi)$ is uniformly
  continuous and thus a morphism of topometric spaces.
\end{lem}
\begin{proof}
  We need to show that for all $\varepsilon > 0$
  there is $\delta > 0$ such that if
  $p,q \in \tS_n(M)$ and $\tilde d_\fp(p,q) < \delta$ then
  for all $\bar b \in M$:
  $|\varphi(\bar x,\bar b)^p - \varphi(\bar x,\bar b)^q| \leq \varepsilon$.
  Indeed, as $\fp$ is a perturbation system for $T$,
  one can find $\delta_1 > 0$ such that
  whenever $\theta \in \Pert_{\fp(\delta_1)}(M,N)$ then
  for all $\bar a,\bar b \in M$:
  $|\varphi(\bar a,\bar b)^M
  - \varphi(\theta(\bar a),\theta(\bar b))^N| \leq \varepsilon/2$.
  By uniform continuity one can find $\delta_2 > 0$ such that if
  $\bar a,\bar a',\bar b \in M$ and
  $d(\bar a,\bar a') < \delta_2$ then
  $|\varphi(\bar a,\bar b)^M - \varphi(\bar a',\bar b)^M|
  \leq \varepsilon/2$.
  Now $\delta = \min \{\delta_1,\delta_2\}$ will do.
\end{proof}

\begin{prp}
  Let $\fp$ be any perturbation system for $T$.
  Then $T$ is stable if and only if $T$ is $\fp$-$\lambda$-stable for
  some $\lambda$, if and only if $T$ is
  $\fp$-$\lambda$-stable for all $\lambda = \lambda^{|\cL|}$.
\end{prp}
\begin{proof}
  Assume that $T$ is stable and let $\lambda = \lambda^{|\cL|}$.
  Then $|\tS_n(M)| = \lambda$ whenever $\|M\| = \lambda$,
  so $T$ is $\fp$-$\lambda$-stable independently of $\fp$.
  In particular, $T$ is $\fp$-$2^{|\cL|}$-stable.

  Conversely, assume $T$ is unstable, say due to an unstable formula
  $\varphi$, and let $\lambda \geq |\cL|$.
  Then there exists $M \models T$ such that
  $\|M\| = \lambda$ and $\|\tS_\varphi(M)\| > \lambda$.
  Since the projection
  $(\tS_n(M),\tilde d_\fp) \to (\tS_\varphi(M),d_\varphi)$ is
  uniformly continuous, it follows that
  $\|(\tS_n(M),\tilde d_\fp)\| > \lambda$,
  and $T$ is not $\fp$-$\lambda$-stable.
\end{proof}

\begin{rmk}
  Recall the various alternatives we considered for the extension of a
  perturbation system $\fp$ to countably many new symbols
  (in this case, constant symbols): the strict variants $\fp_A$ and
  $\fp'_A$ allowed us essentially to fix $A$ (as long as we enumerate
  it judiciously enough), while the relaxed variant $\fp''_A$ only
  considers finite parts of $A$.
  Had we chosen the latter as a basis for extending perturbation
  systems to new parameters, the previous Proposition would fail.
  Indeed, if $T$ is any small theory and $A$ any countable set, then
  $\|(\tS_n(A),\tilde d_{\fp''_A})\| = \aleph_0$, even though $T$ need not
  be stable.
\end{rmk}

\begin{lem}
  Let $\fp$ be a perturbation system for $T$,
  $A \subseteq B \subseteq M \models T$.
  Then the projection map
  $\pi\colon (\tS_n(B),\tilde d_\fp) \to
  (\tS_n(A),\tilde d_\fp)$ is precise (and in particular a
  quotient of topometric spaces).
\end{lem}
\begin{proof}
  Exercise.
\end{proof}

It follows that:
\begin{prp}
  Let $T$ be a countable theory, $\fp$ a perturbation system for $T$.
  Then the collection
  $\{(\tS_n(M'),\tilde d_\fp)\colon M' \preceq M\}$
  is a sufficient family of quotients of
  $(\tS_n(M),\tilde d_\fp)$.
\end{prp}

\begin{thm}
  Let $T$ be a countable theory, $\fp$ a perturbation system for $T$.
  Then the following are equivalent:
  \begin{enumerate}
  \item $T$ is $\fp$-$\lambda$-stable for all $\lambda$.
  \item $T$ is $\fp$-$\aleph_0$-stable.
  \item For every separable model $M$:
    $\|(\tS_n(M),\tilde d_\fp)\| = \aleph_0$.
  \item For every model $M$, the space
    $\|(\tS_n(M),\tilde d_\fp)\|$ is CB-analysable.
  \end{enumerate}
\end{thm}
\begin{proof}
  \begin{cycprf}
  \item[\impnext] Immediate.
  \item[\impnext]
    If $M$ is separable and $A \subseteq M$ is countable and dense, then
    $\tS_n(A) = \tS_n(M)$.
  \item[\impnext] By \fref{cor:CBQuotWeight}.
  \item[\impfirst] By L\"owenheim-Skolem and
    \fref{prp:WeightDenseChar}.
  \end{cycprf}
\end{proof}

\begin{cor}
  Let $T$ be a complete countable theory, $\bar M$ a monster model
  for $T$.
  For a non-empty type-definable set $X \subseteq \bar M^n$ and $\varepsilon > 0$,
  define
  the \emph{$\varepsilon$-Morley rank of $X$ up to $\fp$} by:
  $\RM_{\fp,\varepsilon}(X) = \CB^{(\tS_n(\bar M),\tilde d_\fp)}_{f,\varepsilon}([X])$.
  Then $T$ is $\fp$-$\aleph_0$-stable if and only if
  $\RM_{\fp,\varepsilon}(X)$ is an ordinal for every type-definable set $X$ and
  every $\varepsilon > 0$.
\end{cor}

In particular:
\begin{cor}
  A theory $T$ is $\aleph_0$-stable if and only if
  $\RM_\varepsilon(X)$ is an ordinal for every type-definable set $X$ and
  every $\varepsilon > 0$.
\end{cor}

\providecommand{\bysame}{\leavevmode\hbox to3em{\hrulefill}\thinspace}
\providecommand{\MR}{\relax\ifhmode\unskip\space\fi MR }
\providecommand{\MRhref}[2]{%
  \href{http://www.ams.org/mathscinet-getitem?mr=#1}{#2}
}
\providecommand{\href}[2]{#2}

\end{document}